\documentclass{amsart}

\usepackage{amsmath}
\usepackage{amsthm,amssymb,color,comment}
\usepackage{bbm}
\usepackage{hyperref}
\hypersetup{
    colorlinks=true,
    linkcolor=black,
    filecolor=black,      
    urlcolor=black,
    citecolor=red,
}
\usepackage[TS1,T1]{fontenc}
\usepackage[utf8]{inputenc}
\usepackage{dsfont}
\usepackage{tikz}
\usepackage{enumitem}
\usepackage{stackrel}

\numberwithin{equation}{section}

\newtheorem{thm}{Theorem}[section]
\newtheorem{prop}[thm]{Proposition}
\newtheorem{lem}[thm]{Lemma}
\newtheorem{cor}[thm]{Corollary}

\newtheorem{Def}[thm]{Definition}

\theoremstyle{definition}

\newtheorem{rem}[thm]{Remark}

\DeclareMathOperator{\DIV}{div}

\newcommand{\T}{\mathbb{T}}

\newcommand{\R}{\mathbb{R}}
\newcommand{\N}{\mathbb{N}}

\newcommand{\p}{\partial}
\newcommand{\eps}{\varepsilon}

\newcommand{\divergence}{\text{div}}

\newcommand{\diff}{\mathop{}\!\mathrm{d}}

\def\Xint#1{\mathchoice
	{\XXint\displaystyle\textstyle{#1}}%
	{\XXint\textstyle\scriptstyle{#1}}%
	{\XXint\scriptstyle\scriptscriptstyle{#1}}%
	{\XXint\scriptscriptstyle\scriptscriptstyle{#1}}%
	\!\int}
\def\XXint#1#2#3{{\setbox0=\hbox{$#1{#2#3}{\int}$ }
		\vcenter{\hbox{$#2#3$ }}\kern-.58\wd0}}

\def\dashint{\Xint-}

\makeatletter
\newcommand{\doublewidetilde}[1]{{%
  \mathpalette\double@widetilde{#1}%
}}
\newcommand{\double@widetilde}[2]{%
  \sbox\z@{$\m@th#1\widetilde{#2}$}%
  \ht\z@=.9\ht\z@
  \widetilde{\box\z@}%
}
\makeatother
\author{Dennis Gallenm\"{u}ller}
\address{Institute of Applied Analysis, Ulm University, Helmholtzstra\ss e 18, 89081 Ulm, Germany}
\email{dennis.gallenmueller@uni-ulm.de}
\thanks{}

\author{Piotr Gwiazda}
\address{Institute of Mathematics of Polish Academy of Sciences, Jana i J\k edrzeja \'Sniadeckich 8, 00-656 Warsaw, Poland}
\email{pgwiazda@mimuw.edu.pl}
\thanks{Piotr Gwiazda was supported by National Science Center, Poland through project no. 2018/31/B/ST1/02289.}

\author{Agnieszka Świerczewska-Gwiazda}
\address{Faculty of Mathematics, Informatics and Mechanics, University of Warsaw, Stefana Banacha 2, 02-097 Warsaw, Poland}
\email{aswiercz@mimuw.edu.pl}
\thanks{Agnieszka Świerczewska--Gwiazda was supported by National Science Center, Poland through project no. 2017/27/B/ST1/01569.}

\author{Jakub Woźnicki}
\address{Institute of Mathematics of Polish Academy of Sciences, Jana i J\k edrzeja \'Sniadeckich 8, 00-656 Warsaw, Poland; Faculty of Mathematics, Informatics and Mechanics, University of Warsaw, Stefana Banacha 2, 02-097 Warsaw, Poland}
\email{jw.woznicki@student.uw.edu.pl}
\thanks{Jakub Woźnicki was supported by National Science Center, Poland through project no. 2023/32/O/ST1/03031.}

\begin{document}

\title[From EKP equations to CHKS system.]{Cahn--Hillard and Keller--Segel systems as high-friction limits of Euler--Korteweg and Euler--Poisson equations}

\begin{abstract}
We consider a combined system of Euler--Korteweg and Euler--Poisson equations with friction and exponential pressure with exponent $\gamma > 1$. We show the existence of dissipative measure-valued solutions in the cases of repulsive and attractive potential in Euler--Poisson system. The latter case requires additional restriction on $\gamma$. Furthermore in case of $\gamma \geq 2$ we show that the strong solutions to the Cahn--Hillard--Keller--Segel system are a high-friction limit of the dissipative measure-valued solutions to Euler--Korteweg--Poisson equations.
\end{abstract}


\maketitle

\section{Introduction}

\noindent The present work  focuses  on the high-friction limit of Euler--Poisson and Euler--Korteweg systems. Let us recall that the former is of the form
\begin{align}\label{int:eulerpoisson}
        \left\{\begin{array}{lll}
		\partial_t\rho+\operatorname{div}_x(\rho u)=0,\\
		&\\
		\partial_t(\rho u)+\operatorname{div}_x(\rho u\otimes u)+k\rho u=-\nabla_x\rho^{\gamma} + b\,\rho\nabla_x\Phi_{\rho},\\
		&\\
		-\Delta_x\Phi_{\rho}=\rho-M_{\rho},
		\end{array}\right.
\end{align}
and will converge to the parabolic-elliptic Keller--Segel system
\begin{align}\label{int:kellersegel}
    \left\{\begin{array}{ll}
          \p_t\rho + \DIV_x(-\nabla_x(\rho^\gamma) + b\,\rho\nabla_x\Phi_\rho) = 0\\
         -\Delta_x\Phi_{\rho}=\rho-M_{\rho},
    \end{array}\right.
\end{align}
while the latter consists of following
\begin{align}\label{int:eulerkorteweg}
        \left\{\begin{array}{lll}
		\partial_t\rho+\operatorname{div}_x(\rho u)=0,\\
		&\\
		\partial_t(\rho u)+\operatorname{div}_x(\rho u\otimes u)+k\rho u=-\nabla_x\rho^{\gamma} + \rho\nabla_x(\Delta_x\rho),\\
		\end{array}\right.
\end{align}
which will converge to the Cahn--Hillard equation
\begin{align}\label{int:cahnhillard}
    \p_t\rho - \DIV_x(\nabla_x(\rho^\gamma) - \rho\nabla_x(\Delta_x\rho)) = 0.
\end{align}
In both of the cases the variables are the density $\rho\colon [0,T]\times \T^d\rightarrow [0,\infty)$, the velocity $u\colon [0,T]\times \T^d\rightarrow \R^d$, and the non-local potential $\Phi_{}\colon [0,T]\times \T^d\rightarrow \R$. Here, $\rho\mapsto \rho^{\gamma}$ describes the local pressure with  $\gamma > 1$ fixed. Moreover, for fixed $t\in [0,T]$, the expression 
$$
M_{\rho}:=\underset{\T^d}{\dashint}\rho(t,x)\diff x
$$ 
denotes the spatial average over the density, which is constant in time $t$. The sign of the constant next to the term $\rho\nabla_x\Phi_\rho$ corresponds to attractive or repulsive potentials. We will see, however, that there are differences in analysis of these two possibilities. Let us also remark, that in the case of $b = 0$ the system \eqref{int:eulerpoisson} reduces to the Euler system and \eqref{int:kellersegel} to the porous media equation.\\
\\
Our considerations come from the recent result by Lattanzio and Tzavaras \cite{lattanzio2017fromgas}, where the authors consider a convergence of weak solutions of \eqref{int:eulerpoisson}, \eqref{int:eulerkorteweg} to \eqref{int:kellersegel}, \eqref{int:cahnhillard} respectively. The use of the framework of dissipative measure-valued solutions is the main novelty of the current paper. As we can show rigorously the existence of the dissipative measure-valued solutions, the asymptotic result is not formal anymore, like it was in the case of weak solutions. Overall, interest in diffusive equations, such as \eqref{int:kellersegel} or \eqref{int:cahnhillard} was spurred by papers of Jordan et al. \cite{jordan1998thevariational} as well as of Otto \cite{otto2001thegeometry} in the early two-thousands, where they introduced the use of Wasserstein space of probability measures. In more recent years the high-friction limit has been used to pass from Euler system to porous media equation (see \cite{Huang2005convergence, huang2011l1convergence, lattanzio2013relative}) as well as for other Euler-type systems \cite{carrillo2020relative, feireisl2023onthehighfriction}. An old survey on relaxation limits from hyperbolic to parabolic systems has been done by Donatelli et al. in \cite{donatelli2004convergence}. Let us also note that both Keller--Segel and Cahn--Hillard systems are still a subject of active research, mostly in well-posedeness framework: existence of weak solutions and their asymptotic behavior (see \cite{winkler2019howunstable, bedrossian2011local, chiyo2022anonlinearattraction, cieslak2008finite, elbar2022degenerate, elbar2022fromvlasov, miranville2019thecahn}).\\
\\
The analysis of both systems \eqref{int:eulerpoisson} and \eqref{int:eulerkorteweg} follows similar lines, and thus we will combine them into one  system, which we call Euler--Korteweg--Poisson system (EKP)
\begin{align}
		\left\{\begin{array}{ll}
		\partial_t\rho+\frac{1}{\eps}\operatorname{div}_x(\rho u)=0,\\
		&\\
		\partial_t(\rho u)+\frac{1}{\eps}\operatorname{div}_x(\rho u\otimes u)+\frac{1}{\eps^2}\rho u=-\frac{a}{\eps}\nabla_x\rho^{\gamma}+\frac{b}{\eps}\rho \nabla_x\Phi_{\rho}+\frac{c}{\eps}\rho\nabla_x(\Delta_x\rho),\\
		&\\
		-\Delta_x\Phi_{\rho}=\rho-M_{\rho}\label{eq:EKP}
		\end{array}\right.
\end{align}
for $a, \varepsilon > 0$,  $c \geq 0$, $b\in\R$ as fixed quantities, $T$ denoting the time. The space on which we will consider the aforementioned equation is a flat torus $\T^d$ as to evade unnecessary problems on the boundary ($d = 2,3$ is the dimension). Here, one may notice that the solutions of \eqref{eq:EKP} will depend on $\varepsilon$ which will later converge to $0^+$; moreover, initial conditions $\rho_0$ and $u_0$ will also depend on it. Although we intend later to work almost exclusively on the general Euler--Korteweg--Poisson system, it is more convenient to analyze the main result of the current paper in a less general setting. We provide an asymptotic limit from dissipative measure-valued solutions of hydrodynamic systems to strong solutions of diffusive equations. The existence of dissipative measure-valued solutions to all the cases can be shown, but the question of existence of strong solutions is a more complex matter. To the knowledge of the authors, one can show:
\begin{itemize}
    \item global existence of strong solutions to chemo-repulsive Keller--Segel and porous media equations (Appendix \ref{appendixA}),
    \item local existence of strong solutions to chemo-attractive Keller--Segel equations (see \cite[Lemma 1.2]{cieslak2008finite})
    \item global existence of strong solutions for small initial data to Cahn--Hillard equations (see \cite[Theorem 21]{dang1995stability}).
\end{itemize}
 Before moving forward let us notice that we may reformulate the equations into a much more convenient divergence form.
\begin{rem}
 For strong solutions $(\rho,u,\Phi)\in C^2$ of (\ref{eq:EKP}) we can infer from the Poisson equation that
\begin{align}\label{identityforPhi}
	(\rho-M_{\rho})\nabla_x\Phi=\frac{1}{2}\nabla_x|\nabla_x\Phi|^2-\operatorname{div}_x(\nabla_x\Phi\otimes\nabla_x\Phi).
\end{align}
Moreover, the identity
\begin{align}\label{identityforgradlaplace}
    \rho\nabla_x(\Delta_x\rho)=\nabla_x\left(\frac{1}{2}|\nabla_x\rho|^2+\rho\Delta_x\rho \right)-\operatorname{div}_x(\nabla_x\rho\otimes\nabla_x\rho)
\end{align}
holds.
\end{rem}
\noindent Hence, we can write (\ref{eq:EKP}) equivalently as
\begin{equation}
	\partial_t\rho+\frac{1}{\eps}\operatorname{div}_x(\rho u)=0,
\end{equation}
\begin{equation}
    \begin{aligned}
        &\partial_t(\rho u)+\frac{1}{\eps}\operatorname{div}_x(\rho u\otimes u)+\frac{1}{\eps^2}\rho u=-\frac{a}{\eps}\nabla_x\rho^{\gamma} + \frac{b}{\eps}\operatorname{div}_x\left(\frac{1}{2}|\nabla_x\Phi_\rho|^2\mathbb{I}_d-\nabla_x\Phi_\rho\otimes\nabla_x\Phi_\rho \right)\\
		&+\frac{b}{\eps}\,M_{\rho}\nabla_x\Phi_{\rho} + \frac{c}{\eps}\operatorname{div}_x\left(\frac{1}{2}|\nabla_x\rho|^2\mathbb{I}_d+\rho\Delta_x\rho\mathbb{I}_d-\nabla_x\rho\otimes\nabla_x\rho \right),
    \end{aligned}
\end{equation}
\begin{equation}
    	-\Delta_x\Phi_{\rho}=\rho-M_{\rho}.
\end{equation}
Here, we also introduced the notation $\mathbb{I}_{d}$ for the identity matrix of dimension $d$.\\
\\
\noindent Our main aim is to show that from \eqref{eq:EKP} we may converge to the combined system of \eqref{int:kellersegel} and \eqref{int:cahnhillard}, which we call Cahn--Hillard--Keller--Segel system (CHKS). This means, that as $\eps$ goes to zero we obtain in a limit
\begin{equation}
	\begin{aligned}
		\partial_t\rho - \operatorname{div}_x(a\nabla_x(\rho^\gamma) + b\rho\nabla_x \Phi_{\rho} - c\rho\nabla_x(\Delta_x\rho))&=0,\\
		-\Delta_x\Phi_{\rho}&=\rho-M_{\rho}.\label{eq:CHKS}
	\end{aligned}
\end{equation}
\noindent For the passage to the limit system, we use the highly efficient approach of the relative entropy method, which found use in a various different fields, ranging from weak-strong uniqueness problems \cite{woznicki2022weakstrong, wiedemann2017fluid, brenier2011weak, gwiazda2015weakstrong}, to stability studies, asymptotic limits, dimension reduction \cite{bella2014dimension, christoforou2018relative, feireisl2012relative, giesselmann2014singular, giesselemann2017stability}.\\
\\
Finally, in accordance to the discussion above about the notion of strong solutions, we want to make a distinction between two cases.
\begin{Def}\label{def:strongsolutions}
    We will say that $(r, \Phi_r)$ is a strong solution to the system \eqref{eq:CHKS} if
    \begin{itemize}
        \item $(r, \Phi_r)\in C^{5, 1}(\T^d \times [0, T])$, whenever $c > 0$,
        \item $(r, \Phi_r)\in C^{2, 1}(\T^d\times [0, T])$, whenever $c = 0$.
    \end{itemize}
\end{Def}
\noindent The structure of the paper is as follows. In Section 2 we recall the notion of Young measures and provide the definition of dissipative measure-valued solutions. In Section 3 we state the main results: existence of dissipative measure-valued solutions to EKP equations and their convergence in high-friction limit to strong solutions CHKS system. In Section 4 we provide the proof to the former, and in Section 5 to the latter.

\section{Dissipative measure-valued solutions.}

\noindent Let us start this section with a review of the results on Young measures. For details, motivations and examples of particular applications, we refer the reader to papers by Young \cite{young1942generalizedsurfaces, young1942generalized}, Ball \cite{ball1989version} as well as modern reviews \cite[Chapter 3]{muller1999variational}, \cite[Chapter 6]{pedregal2012parametrized} and \cite[Chapter 4]{rindler2018calcvar}.\\

\noindent We start with the most important result that we recall from \cite[Theorem 6.2]{pedregal2012parametrized}.
\begin{thm}\textup{\textbf{(Fundamental Theorem of Young Measures)}}\label{app:fundamental_thm}\\
Let $\Omega \subset \R^n$ be a measurable set and let $z_j: \Omega \to \R^m$ be measurable functions such that 
\begin{equation}\label{app:cond_mass_esc_toinfty}
\sup_{j \in \N} \int_{\Omega} g(|z_j(x)|) \, \diff x < \infty
\end{equation}
for some continuous, nondecreasing function $g:[0,\infty) \to [0,\infty)$ with $\lim_{t\to \infty} g(t) = \infty$. Then, there exists a subsequence (not relabeled) and a weakly-$\ast$ measurable family of probability measures $\nu = (\nu_x)_{x \in \Omega}$ with the property that whenever the sequence $\{\psi(x,z_j(x))\}_{j \in \N}$ is weakly compact in $L^1(\Omega)$ for a Carath\'eodory function $\psi: \Omega \times \R^m \to \R$, we have
\begin{equation}\label{app:weaklimitL1}
\psi(x,z_j(x)) \rightharpoonup \int_{\R^m} \psi(x,\lambda) \, \diff \nu_x(\lambda) \qquad \mbox{ weakly in } L^1(\Omega).
\end{equation}
We say that the sequence $\{z_j\}_{j \in \N}$ generates the family of Young measures $\{\nu_x\}_{x\in \Omega}$. 
\end{thm}
\begin{rem}
    For further reference we introduce the notation
    $$
    \langle\nu_x, \psi\rangle := \int_{\R^m} \psi(x,\lambda) \, \diff \nu_x(\lambda).
    $$
\end{rem}
\begin{rem}
Condition \eqref{app:cond_mass_esc_toinfty} prevents the mass to accumulate at infinity. In \cite[Chapter 3]{muller1999variational} it was formulated as 
$$
\lim_{M \to \infty} \sup_{j \in \N} \left| \left\{ x\in \Omega: \, |z_j(x)|\geq M \right\} \right| = 0.
$$
Similarly, in \cite[Chapter 6]{pedregal2012parametrized} the condition \eqref{app:cond_mass_esc_toinfty} was replaced with boundedness of the sequence $\{z_j\}_{j\in \N}$ in $L^p(\Omega)$ for some $1\leq p \leq \infty$ which is slightly weaker.
\end{rem}
\begin{rem}\label{app:weak_lim_identification}
Let $\{z_j\}_{j \in \Omega}$ be a sequence converging weakly to some $z$ in $L^p(\Omega; \R^m)$ where $\Omega$ is a bounded domain and $1 < p \leq \infty$ (weakly-$\ast$ in case $p=\infty$). If we let $\psi(x,u) = u_i$ (i.e. $i$-th component of $u \in \R^m$) for $i = 1, ..., m$, we have that $\{\psi(x,z_j)\}_{j \in \N}$ is weakly compact in $L^1(\Omega)$ and thus by Theorem \ref{app:fundamental_thm} we obtain
$$
z_j(x) \rightharpoonup \int_{\R^m} \lambda \diff \nu_{x}(\lambda) = z(x),
$$
where the latter equality follows by uniqueness of weak limits. More generally, suppose that $\psi$ is such that $\{\psi(x,z_j)\}_{j \in \N}$ is bounded in $L^p(\Omega; \R^m)$ for $1 < p \leq \infty$. Then, up to a subsequence,
\begin{equation}\label{app:weakconv_car_lp}
\psi(x,z_j) \rightharpoonup \int_{\R^m} \psi(x,\lambda) \diff \nu_{x}(\lambda),
\end{equation}
where the weak convergence is replaced by the weak-$\ast$ one in case $p = \infty$.
\end{rem}
\noindent Often times one cannot guarantee the $L^p$ bound of the sequence with $p > 1$ and the concentrations might appear. We cover a very specific, but sufficient for this work case in the next lemma.
\begin{lem}\label{app:concentration}
Suppose that $f$ is a continuous function and $\{f(z_j)\}_{j\in\N}$ is bounded in $L^\infty((0, T); L^1(\Omega))$, then there exists $m^f\in L^\infty((0, T); \mathcal{M}(\Omega))$, such that (up to the subsequence)
$$
f(z_j) - \langle\nu_{t, x}, f\rangle \overset{*}{\rightharpoonup} m^f \qquad\text{ weakly-$\ast$ in }L^\infty((0, T); \mathcal{M}(\Omega)).
$$
\end{lem}
\begin{proof}
From the natural embedding
$$
L^\infty((0, T); L^1(\Omega)) \hookrightarrow L^\infty((0, T); \mathcal{M}(\Omega))
$$
and the Banach--Alaoglu Theorem we obtain
$$
f(z_j) \overset{*}{\rightharpoonup} f_\infty \qquad\text{ weakly-$\ast$ in }L^\infty((0, T); \mathcal{M}(\Omega))
$$
up to the subsequence. Setting $m^f = f_\infty - \langle\nu_{t,x}, f\rangle$ completes the proof.
\end{proof}
\noindent As a short-hand notation we will write
\begin{align}
    \overline{f}(t, x) = \langle\nu_{t, x}, f(\lambda)\rangle + m^f(t)(\diff x).
\end{align}
\noindent Note that, the concentration measure $m^f \equiv 0$ if the sequence $\{f(z_j)\}$ is weakly relatively precompact in $L^1(\Omega)$. In our case, we shall use dummy variables $\lambda = (s, v, F, G) \in [0,\infty) \times \R^d\times \R^d\times \R^d$ to express the ones in the equation. As an example 
\begin{align*}
\overline{\rho} &= \langle\nu, s\rangle \\
\overline{\rho u} &= \langle\nu, sv\rangle + m^{\rho u}\\
\overline{\nabla_x\Phi_{\overline{\rho}}} &= \langle\nu, F\rangle\\
\overline{\rho\nabla_x\rho} &= \langle\nu, sG\rangle + m^{\rho\nabla_x\rho}
\end{align*}
and similarly for other terms. To establish some inequalities for relative entropies we also need the following proposition.
\begin{prop}\label{app:concetrationmeasuresbounding}\textup{\textbf{(\cite{gwiazda2020dissipative}, Proposition 3.3})}\\
Let $\nu_x$ be a Young measure generated by the sequence $\{z_j\}_{j\in\N}$. If two continuous functions \mbox{$f_1: \R^m \rightarrow \R^k$}, $f_2:\R^m\rightarrow \R$ satisfy $|f_1(z)|\leq f_2(z)$ for any $z\in \R^m$ and if $\{f_2(z_j)\}_{j\in\N}$ is uniformly bounded in $L^1(\Omega)$, then we have
$$
|m^{f_1}|(A) \leq m^{f_2}(A)
$$
for any Borel set $A\subset\Omega$.
\end{prop}
\noindent The proposition above might be used, for instance, to establish the following bounds
\begin{align*}
    |m^{\rho u}|&\leq m^\rho + m^{\rho|u|^2}\\
    |m^{\rho u\otimes u}|&\leq m^{\rho|u|^2} 
\end{align*}
and the similar ones appearing later in the article. Now we may proceed to define the desired notion of solutions.

\begin{Def}\label{def:measurevaluedsolution}
	A measure $(\nu, m)$ is a dissipative measure-valued solution of the Euler--Korteweg--Poisson system (\ref{eq:EKP}) with measure-valued initial data $(\nu_0,m_0)$ if
	$$
	m=\left(m^{\rho u},m^{\rho u\otimes u},m^{|\nabla_x \Phi_{\rho}|^2},m^{\nabla_x\Phi_\rho\otimes \nabla_x\Phi_\rho}, m^{\rho^{\gamma}}, m^{|\nabla_x\rho|^2}, m^{\rho\nabla_x\rho}, m^{\nabla_x\rho\otimes\nabla_x\rho}\right)
	$$ 
	and
	\begin{align*}
		\nu&\in L^{\infty}_{\operatorname{w}}\big((0,T)\times \T^d;\mathcal{M}^+\big([0,\infty)\times \R^d\times \R^d\times\R^d\big)\big),\\
		m^{\rho^{\gamma}},\ m^{|\nabla_x\Phi_\rho|^2},\ m^{|\nabla_x\rho|^2}&\in L^{\infty}\big((0,T);\mathcal{M}^+(\T^d)\big),\\
		m^{\rho u},\ m^{\rho\nabla_x\rho}&\in L^{\infty}\big((0,T);\mathcal{M}(\T^d)^d\big),\\
		m^{\rho u\otimes u},\ m^{\nabla_x\Phi_\rho\otimes \nabla_x\Phi_\rho},\ m^{\nabla_x\rho\otimes\nabla_x\rho}&\in L^{\infty}\big((0,T);\mathcal{M}(\T^d)^{d\times d}\big),
	\end{align*}
	as well as the following relations hold 
	\begin{itemize}
	\item[(a)]\textit{Mass equation:}
	\begin{align}
		\int_{\T^d}^{}\overline{\rho}\psi(\tau,\cdot)\diff x-\int_{\T^d}^{}\overline{\rho_0}\psi(0,\cdot)\diff x=\int_{0}^{\tau}\int_{\T^d}^{}\overline{\rho}\partial_t\psi+\frac{1}{\eps}\overline{\rho u}\cdot\nabla\psi\,\diff x\diff t\label{eq:massmvs}
	\end{align}
	for a.e.~$\tau\in(0,T)$ and all $\psi\in C^1([0,T]\times \T^d)$.\\
	\item[(b)]\textit{Momentum equation:}
	\begin{equation}
		\begin{aligned}
			&\int_{\T^d}^{}\overline{\rho u}\cdot \varphi(\tau,\cdot)\diff x-\int_{\T^d}^{}\overline{\rho_0u_0}\cdot\varphi(0,\cdot)\diff x=\int_{0}^{\tau}\int_{\T^d}^{}\overline{\rho u}\cdot\partial_t\varphi+\frac{1}{\eps}\overline{\rho u\otimes u}:\nabla_x\varphi-\frac{1}{\eps^2}\overline{\rho u}\cdot \varphi\diff x\diff t\\
			&+\int_0^\tau\int_{\T^d}\frac{a}{\eps}\overline{\rho^{\gamma}}\operatorname{div}_x\varphi\,\diff x\diff t - \int_{0}^{\tau}\int_{\T^d}^{}\frac{b}{2\eps}\overline{|\nabla_x\Phi_{\overline{\rho}}|^2}\operatorname{div}_x\varphi+\frac{b}{\eps}\overline{\nabla_x\Phi_{\overline{\rho}}\otimes\nabla_x\Phi_{\overline{\rho}}}:\nabla_x\varphi\diff x\diff t\\
			&+\int_0^\tau\int_{\T^d}\frac{b}{\eps}M_{\overline{\rho}}\overline{\nabla_x\Phi_{\overline{\rho}}}\cdot\varphi\,\diff x\diff t+\int_{0}^{\tau}\int_{\T^d}^{}\frac{c}{2\eps}\overline{|\nabla_x\rho|^2}\operatorname{div}_x\varphi+\frac{c}{\eps}\overline{\rho\nabla_x\rho}\nabla_x(\operatorname{div}_x\varphi)\diff x\diff t\\
			&+\int_0^\tau\int_{\T^d}\frac{c}{\eps}\overline{\nabla_x\rho\otimes\nabla_x\rho}:\nabla_x\varphi\,\diff x\diff t\label{eq:momentummvs}
		\end{aligned}
	\end{equation}
	for a.e.~$\tau\in(0,T)$ and all $\varphi\in C^1\big([0,T]\times \T^d;\R^d\big)$.\\
	\item[(c)]\textit{Poisson equation:} For a.e.~$\tau\in(0,T)$ and all $\chi\in C^1(\T^d)$ it holds that
	\begin{align}
	    \int_{\T^d}^{}\langle\nu_{(\tau,\cdot)},F\rangle\cdot\nabla_x\chi\,\diff x=\int_{\T^d}^{}\overline{\rho}\,\chi\,\diff x-M_{\overline{\rho}}\,\int_{\T^d}^{}\chi\,\diff x.\label{eq:poissonmvs}
	\end{align}
	Moreover, there exists some $\Phi_{\overline{\rho}}\in L^2\big((0,T);W^{2,2}(\T^d)\big)$ such that
	\begin{align}
	    \nabla_x\Phi_{\overline{\rho}}=\langle\nu,F\rangle\label{eq:nablaPhiidentity}
	\end{align}
	a.e.~on $(0,T)\times \T^d$.
	\item[(d)]\textit{Energy inequality}: There exists a measure $m^{\rho^{}|u^{}|^2}\in L^{\infty}\big((0,T);\mathcal{M}^+(\T^d)\big)$ such that
	\begin{equation}
	    \begin{aligned}
	        &\int_{\T^d}^{}\left(\frac{1}{2}\overline{\rho^{}|u^{}|^2}+\frac{a}{\gamma-1}\overline{\rho^{\gamma}}-\frac{b}{2}\overline{|\nabla\Phi^{}|^2}+\frac{c}{2}\overline{|\nabla\rho^{}|^2}\right)(\tau, x)\diff x+\frac{1}{\eps^2}\int_{0}^{\tau}\int_{\T^d}^{}\overline{\rho |u|^2}\,\diff x\diff t\\
	        &\leq \int_{\T^d}^{}\frac{1}{2}\overline{\rho_0^{}|u_0|^2}+\frac{a}{\gamma-1}\overline{\rho_0^{\gamma}}-\frac{b}{2}\overline{|\nabla\Phi_0^{}|^2}+\frac{c}{2}\overline{|\nabla\rho_0^{}|^2}\,\diff x\label{eq:energyinequalitymvs}
	    \end{aligned}
    \end{equation}
	for a.e.~$\tau\in(0,T)$.
% 	\item[(e)]\label{approximationassumption}\textit{Approximation}: There exists an approximation $(\rho_j, u_j, \Phi_j)$ generating the pairing $(\nu, m)$, such that both $\rho_j, \Phi_j\in L^2((0, T); W^{1, 2}(\T^d))$, $u_j \in L^2((0, T)\times \T^d)$ and
% 	\begin{equation}\label{distributionalpoissonsapprox}
% 	    -\Delta_x\Phi_j + H_j = \rho_j - M_{\rho_j} 
% 	\end{equation}
% 	in the sense of distributions, where $H_j\in L^2((0, T); W^{-1, 2}(\T^d))$ is an error term for which
% 	\begin{align}
% 	    H_j \overset{*}{\rightharpoonup} 0\text{ in }L^2((0, T); W^{-1, 2}(\T^d))
% 	\end{align}
	\end{itemize}
	All the above integrals have to exist as part of the definition. In particular, if $c = 0$, then
 $$
 \overline{\rho} := \langle\nu_{t,x}, s\rangle \in L^\infty((0, T); L^\gamma(\T^d))
 $$
 and if $c \neq 0$, then
 \begin{align}\label{def:sobolevregularitydensity}
 \overline{\rho} := \langle\nu_{t, x}, s\rangle \in L^\infty((0, T); W^{1, 2}(\T^d)).
 \end{align}
\end{Def}

\noindent Before stating the main theorems, we recall a property of dissipative measure-valued solutions, that is going to be necessary while treating the case $b > 0$.

\begin{lem}\label{lem:boundfornegativebees}\textup{\textbf{(Lemma 3.7, \cite{lattanzio2017fromgas})}}\\
Suppose $\gamma > 2 - \frac{2}{d}$, $(\overline{\rho}, \Phi_{\overline{\rho}})$ are as in Definition \ref{def:measurevaluedsolution} and $(r, \Phi_r)$ are as in Definition \ref{def:strongsolutions}. Then, there exists a constant $K > 0$ (independent of $\varepsilon)$ such that 
$$
\int_{\T^d}\overline{|\nabla_x \Phi_{\overline{\rho}} - \nabla_x \Phi_r|^2}\diff x \leq K\int_{\T^d}h(\overline{\rho}|r)\diff x,
$$
where 
\begin{align}\label{lem:definitionofh}
h(r) := \frac{a}{\gamma - 1}r^\gamma, \qquad h(\overline{\rho}|r) := h(\overline{\rho}) - h(r) - h'(r)(\overline{\rho} - r).
\end{align}
\end{lem}

\section{Statement of the main results.}

\noindent We state the main results in the theorems below.
\begin{thm}\label{thm:existenceofmeasurevaluedsolutions}
    Suppose initial conditions $(\rho_0, u_0)\in L^\infty(\T^d)$ and one of the conditions below holds
    \begin{itemize}
        \item $\frac{2}{K} > b > 0$, $\gamma > 2 - \frac{2}{d}$, $c\geq 0$, $a, \varepsilon >0$ (here, $K$ is from Lemma \ref{lem:boundfornegativebees}),
        \item or $b \leq 0$, $\gamma > 1$, $c\geq 0$, $a, \varepsilon >0$.
    \end{itemize}
Then, the dissipative measure-valued solution to \eqref{eq:EKP} as in Definition \ref{def:measurevaluedsolution} exists.
\end{thm}

\begin{thm}\label{theo:poissonlimit}
	Let $r_0\in C^3(\T^d)$ be initial data giving rise to a strong solution $(r,\Phi)$ as in Definition \ref{def:strongsolutions} of (\ref{eq:CHKS}) satisfying $r\geq \rho_*$ on $[0,T]\times \T^d$ for some $\rho_*>0$, and for all $\eps>0$ let $(\nu^{\eps}_0,m^{\eps}_0)$ be well-prepared measure-valued initial data in the sense that
	\begin{align*}
		\mathcal{E}^{\eps}_{\operatorname{rel}}(0)\rightarrow 0.
	\end{align*}
	Suppose that $\frac{2}{K} > b$, $\gamma \geq 2$, $c\geq 0$, $a, \varepsilon >0$ (here, $K$ is from the Lemma \ref{lem:boundfornegativebees}). Then, for all $\tau\in [0,T]$ and for any high friction sequence of dissipative measure-valued solutions $(\nu^{\eps},m^{\eps})$ of (\ref{eq:EKP}) with initial data $(\rho^{\eps}_0,u^{\eps}_0)$ it holds that
	\begin{align}
		\mathcal{E}^{\eps}_{\operatorname{rel}}(\tau)+\frac{1}{2\eps^2}\int_{0}^{\tau}\int_{\T^d}^{}\overline{\rho^{\eps}|u^{\eps}-U^{\eps}|^2}\diff x\diff t\rightarrow 0\text{ as }\eps\rightarrow 0.\label{eq:relativeentropyconvergence}
	\end{align}
	Here, %for $t\in [0,T]$ 
	the relative entropy
	\begin{align*}
		\mathcal{E}^{\eps}_{\operatorname{rel}}(\tau):=\int_{\T^d}^{}\frac{1}{2}\overline{\rho^{\eps}|u^{\eps}-U^{\eps}|^2}+\overline{h(\rho^{\eps}|r)}-\frac{b}{2}\overline{|\nabla_x\Phi_{\rho^\eps}-\nabla_x\Phi_r|^2}+\frac{c}{2}\overline{|\nabla_x\rho^\eps-\nabla_x r|^2}\diff x,
	\end{align*}
	is considered, where
	\begin{align}\label{definitionofbigueps}
		 U^{\eps}:=-\eps\nabla_x(h'(r)-b\Phi_r-c\Delta_x r), 
	\end{align}
	and
	$$
	h(r), \qquad h(\rho^\eps|r) \qquad \text{ are defined by }\eqref{lem:definitionofh}.
	$$
	As a consequence of (\ref{eq:relativeentropyconvergence}), we obtain that up to a subsequence
	\begin{align*}
		m^{|\nabla_x\Phi_{\rho^\eps}|^2}\overset{*}{\rightharpoonup}&\,0\text{ weakly$-\ast$ in }L^{\infty}\big((0,T);\mathcal{M}(\T^d)\big),\\
		m^{\nabla_x\Phi_{\rho^\eps}\otimes \nabla_x\Phi_{\rho^\eps}}\overset{*}{\rightharpoonup}&\,0\text{ weakly$-\ast$ in }L^{\infty}\big((0,T);\mathcal{M}(\T^d)\big),\\
		m^{|\nabla_x\rho^\eps|^2}\overset{*}{\rightharpoonup}&\,0\text{ weakly$-\ast$ in }L^{\infty}\big((0,T);\mathcal{M}(\T^d)\big),\\
		m^{\nabla_x\rho^\eps\otimes \nabla_x\rho^\eps}\overset{*}{\rightharpoonup}&\,0\text{ weakly$-\ast$ in }L^{\infty}\big((0,T);\mathcal{M}(\T^d)\big),
	\end{align*}
	and
	\begin{align*}
		\overline{\rho^{\eps}}=\langle\nu^{\eps},\rho\rangle\overset{*}{\rightharpoonup}&\,r\text{ weakly$-\ast$ in }L^{\infty}\big((0,T); L^{\gamma}(\T^d)\big),\\
		b\overline{\nabla_x\Phi_{\rho^\eps}}=b\langle\nu^{\eps},F\rangle\overset{*}{\rightharpoonup}&\,b\,\nabla_x\Phi_r\text{ weakly$-\ast$ in }L^{\infty}\big((0,T);L^2(\T^d)^d\big),\\
		c\overline{\nabla_x\rho^\eps}=c\langle\nu^{\eps},G\rangle\overset{*}{\rightharpoonup}&\,c\,\nabla_x r\text{ weakly$-\ast$ in }L^{\infty}\big((0,T);L^2(\T^d)^d\big).
	\end{align*}
\end{thm}
\noindent as $\varepsilon\to 0^+$. 
\begin{rem}
The convergence of the density $\rho$ can be better characterized. Indeed, one can show that 
\begin{align}\label{convergencerhoinwasserstein}
\nu^\varepsilon(\diff s) \rightarrow \delta_{r}\quad \text{ in } L^\infty((0, T); L^{\mathrm{min\{\gamma, 2\}}}(\T^d; (\mathcal{P}([0, +\infty)), W_{\min\{\gamma , 2\}}))).
\end{align}
Moreover, if we consider a Young measure generated by the momentum $m = \rho u$ (meaning that $\overline{m} = \langle\nu^\varepsilon, w\rangle + m^{\rho u}$), then
\begin{align}\label{convergencemomentuminwasserstein}
\nu^\varepsilon(\diff w) \rightarrow \delta_{0}\quad \text{ in } L^\infty((0, T); L^1(\T^d; (\mathcal{P}(\R^d), W_1))),
\end{align}
where $(\mathcal{P}([0, +\infty)), W_{\min\{\gamma , 2\}})$ and $(\mathcal{P}(\R^d), W_1)$ denote the spaces of probability measures with Wasserstein distance. We recall the following property of function $h$ and its relative counterpart.
\begin{lem}\label{ineq:inequalityforhrhor}\textup{\textbf{(Lemma 3.4 and Remark 3.5, \cite{lattanzio2017fromgas})}}
Let $h(\rho) = \frac{1}{\gamma - 1}\rho^\gamma$ with $\gamma > 1$. Suppose, that $r\in [a, b]$ with $a > 0$ and $b < +\infty$, then there exists $R_0$ dependent on the interval $[a, b]$ and constants $C_1$ and $C_2$ (dependent on $R_0$ and $[a, b]$) such that
\begin{align}
    h(\rho | r) \geq \left\{\begin{array}{ll}
       C_1|\rho - r|^\gamma , \quad &\text{ if } \rho > R_0, r\in [a, b] \\
       C_2|\rho - r|^2 ,\quad  &\text{ if } 0<\rho\leq R_0, r\in [a, b] 
    \end{array}
    \right.
\end{align}
Moreover if $\gamma \geq 2$, then there exists a constant $C_3$ such that for any $\rho\in\R_+$ and $r\in [a,b]$
\begin{align}\label{ineq:wassersteinsquare}
    h(\rho|r) \geq C_3 |\rho - r|^2.
\end{align}
\end{lem}
\begin{cor}
Under the assumption of Lemma \ref{ineq:inequalityforhrhor}, for any $\delta > 0$, there exists a constant $C(\delta) > 0$, such that
\begin{align}\label{ineq:wassersteindistancecruacialineq}
    |\rho - r|^\gamma \leq \delta + C(\delta)h(\rho|r).
\end{align}
\end{cor}
\noindent Using \eqref{ineq:wassersteinsquare}, \eqref{ineq:wassersteindistancecruacialineq} and Theorem \ref{theo:poissonlimit} one may notice
\begin{equation}\label{ineq:wassersteinconvergencerho}
    \begin{split}
        0\leq \int_{\T^d}\int_0^{\infty}|s - r|^{\min\{\gamma, 2\}}\nu^\varepsilon(\diff s)\diff x \leq |\T^d|\delta + C(\delta)\int_{\T^d}\overline{h(\rho^\varepsilon|r)}\diff x.
    \end{split}
\end{equation}
Converging first with $\varepsilon\to 0^+$ and then with $\delta\to 0^+$, gives the \eqref{convergencerhoinwasserstein}. Now notice that from \eqref{ineq:wassersteinconvergencerho} we can deduce
$$
\int_{\T^d}\int_0^\infty s\nu^\varepsilon(\diff s)\diff x \leq C,
$$
where $C$ is independent of $\varepsilon$. Hence, we use Young's inequality and get
\begin{equation}\label{inequalitywassersteinmomentum}
    \begin{split}
        0&\leq \int_{\T^d}\int_{\R^d}|w|\nu^\varepsilon(\diff w)\diff x\leq \delta\int_{\T^d}\int_0^\infty s\,\nu^\varepsilon(\diff s)\diff x + \frac{1}{4\delta}\int_{\T^d}\int_{0}^\infty\int_{\R^d}\frac{|w|^2}{s}\nu^\varepsilon(\diff s, \diff w)\diff x\\
        &\leq \delta\, C + \frac{1}{2\delta}\left(\int_{\T^d}\overline{\rho^\varepsilon\left|\frac{m^\varepsilon}{\rho^\varepsilon} - U^\varepsilon\right|^2}\diff x + \int_{\T^d}\overline{\rho^\varepsilon}|U^\varepsilon|^2\diff x\right).
    \end{split}
\end{equation}
From the definition of $U^\varepsilon$ (see \eqref{definitionofbigueps})
\begin{align*}
    \int_{\T^d}\overline{\rho^\varepsilon}|U^\varepsilon|^2\diff x \lesssim \varepsilon^2 \int_{\T^d}\overline{\rho^\varepsilon}\diff x\lesssim \varepsilon^2,
\end{align*}
and from Theorem \ref{theo:poissonlimit}
\begin{align*}
    \int_{\T^d}\overline{\rho^\varepsilon\left|\frac{m^\varepsilon}{\rho^\varepsilon} - U^\varepsilon\right|^2}\diff x \longrightarrow 0.
\end{align*}
Thus, converging with $\varepsilon\to 0^+$ and then with $\delta\to 0^+$ in \eqref{inequalitywassersteinmomentum} gives the \eqref{convergencemomentuminwasserstein}.
\end{rem}
% \begin{rem}
% The last assumption, although it might seem technical and artificial, is a natural extension of the DiPerna and Majda model of generalized Young measures. The discussion on this is carried out in the Appendix \ref{app:dipernamajda}.
% \end{rem}
\section{Existence of dissipative measure-valued solutions to EKP system.}\label{section:proofofexistence}
\noindent Here, we state the proof of the Theorem \ref{thm:existenceofmeasurevaluedsolutions}. One may notice that the beginning of the proof is indifferent to the choice of attractive or repulsive potential. This distinction becomes important in the proof of the Lemma \ref{lem:compactnessforgalerkin} and later on, while converging from approximate solutions (defined below) to the dissipative measure-valued ones. A similar argument has been made in \cite{carrillo2021dissipative}. For now, we omit the index $\varepsilon$.
    \subsection{Approximate solutions}Fix $\mu > 0$ and let $((\cdot ; \cdot)) = (\cdot; \cdot)_{W^{6, 2}_0(\T^d)}$ denote the standard scalar product in $W^{6, 2}_0(\T^d)$. We will first seek the approximate solution to our problem, i.e.  the triple $(u^\mu, \rho^\mu, \Phi_{\rho^\mu})$ such that
    \begin{align}
        \int_0^\tau\int_{\T^d} \rho^\mu\p_t\psi(t, x) + \frac{1}{\eps}\rho^\mu u^\mu\,\nabla_x\psi(t, x)\diff x\diff t = \int_{\T^d}\rho^\mu(\tau, x)\psi(\tau, x) - \rho_0^\mu(x)\psi(0, x)\diff x \label{approximateEKP1}
    \end{align}
    for any $\psi\in C^1([0, T]\times \T^d)$,
    \begin{equation}
    \begin{aligned}
        &\int_0^\tau\int_{\T^d}\rho^\mu u^\mu\p_t\phi + \frac{1}{\eps}\rho^\mu u^\mu\otimes u^\mu : \nabla_x\phi - \frac{1}{\eps^2}\rho^\mu\,u^\mu\,\phi\diff x\diff t+ \frac{a}{\eps}\int_0^\tau\int_{\T^d}(\rho^\mu)^\gamma \divergence_x\phi\diff x\diff t\\
        &+\int_0^\tau\int_{\T^d}\left[\frac{b}{\eps}\nabla_x\Phi_{\rho^\mu}\otimes \nabla_x\Phi_{\rho^\mu} + \frac{b}{2\eps}|\nabla_x\Phi_{\rho^\mu}|^2\mathbb{I}_d\right] : \nabla_x\phi\diff x\diff t- \frac{b}{\eps}M_{\rho^\mu}\int_0^\tau\int_{\T^d}\nabla_x\Phi_{\rho^\mu}\,\phi\diff x\diff t\\
        & + \int_0^\tau\int_{\T^d}\left[\frac{c}{2\eps}|\nabla_x\rho^\mu|^2\mathbb{I}_d + \frac{c}{\eps}\nabla_x\rho^\mu\otimes\nabla_x\rho^\mu \right] : \nabla_x\phi\diff x\diff t + \frac{c}{\eps}\int_0^\tau\int_{\T^d}\rho^\mu\nabla_x\rho^\mu\cdot\nabla_x(\divergence_x\phi)\diff x\diff t\\
        &= \mu\int_0^\tau ((u^\mu; \phi))\diff t + \int_{\T^d}(\rho^\mu u^\mu)(\tau, x)\phi(\tau, x) - \int_{\T^d}\rho^\mu_0 u^\mu_0\phi(0, x)\diff x
    \end{aligned}
    \end{equation}
    for every $\phi\in C^2([0, T]\times\T^d)$ and
    \begin{align}
        \int_{\T^d}\nabla_x\Phi_{\rho^\mu}(\tau, x)\nabla_x\chi(x)\diff x = \int_{\T^d}\rho^\mu(\tau, x)\chi(x)\diff x - M_{\rho^\mu}\int_{\T^d}\chi(x)\diff x \label{approximateEKP3}
    \end{align}
    for any $\chi \in C^1(\T^d)$. Here 
    $$
    M_{\rho^\mu} = \underset{\T^d}{\dashint}\rho^\mu(t,x)\diff x,
    $$
    which is a constant in time for each $\mu > 0$. Initial conditions $\rho^\mu_0$ and $u^\mu_0$ are defined by $\rho^\mu_0 = \rho_{0, \mu} + \mu$ and $u^\mu_0 = u_{0, \mu}$, where $\rho_{0,\mu}$ and $u_{0, \mu}$ denote standard mollifications of $\rho_0$ and $u_0$ respectively.\\
    \\
    \noindent To this end, we will look at the Galerkin approximations. For the sake of convenience and avoidance of the overload of indexes, we drop the index $\mu$. Let $\{\omega_i\}_{i\in\mathbb{N}}$ be the sequence of smooth functions, which is the orthogonal basis of $W^{6, 2}_0(\T^d)$ and orthonormal one of $L^2(\T^d)$ (for the existence of such a basis see Theorem 4.11, \cite{malek1996weak}, p. 290). Let
    $$
    u^n(t, x) = \sum_{i = 1}^n c_i^n(t)\omega_i(x).
    $$
    We say that the triple $(u^n, \rho^n, \Phi_{\rho^n})$ is the Galerkin approximation of the problem \eqref{approximateEKP1} - \eqref{approximateEKP3} if it solves for all $i = 1, ..., n$ 
    \begin{align}
        \p_t\rho^n + \frac{1}{\eps}\divergence_x(\rho^n u^n) = 0 \label{galerkingEP1}
    \end{align}
    \begin{equation}
    \begin{aligned}
        &\int_{\T^d}\rho^n \p_t u^n\omega_i + \frac{1}{\eps}\rho^n\nabla_x u^n\cdot u^n\omega_i + \frac{b}{\eps}\divergence_x \left[\nabla_x\Phi_{\rho^n}\otimes \nabla_x\Phi_{\rho^n} - \frac{1}{2}|\nabla_x\Phi_{\rho^n}|^2\mathbb{I}_d\right] \cdot \omega_i\diff x\\
        & + \int_{\T^d}\frac{c}{\eps}\divergence_x\left[\frac{1}{2}|\nabla_x\rho^n|^2\mathbb{I}_d + \nabla_x\rho^n\otimes\nabla_x\rho^n \right] \cdot \omega_i\diff x
         - \frac{c}{\eps}\int_{\T^d}\rho^n\nabla_x\rho^n\cdot\nabla_x(\divergence_x\omega_i)\diff x\\
         &+ \frac{1}{\eps^2}\int_{\T^d}\rho^n u^n \omega_i\diff x + \frac{a}{\eps}\int_{\T^d}\nabla_x((\rho^n)^\gamma) \omega_i\diff x - \frac{b}{\eps}M_{\rho^n}\int_{\T^d}\nabla_x\Phi_{\rho^n}\omega_i\diff x + \mu ((u^n; \omega_i)) = 0 \label{galerkinEP2}
    \end{aligned}
    \end{equation}
    \begin{align}
        -\Delta \Phi_{\rho^n} = \rho^n - M_{\rho^n}\label{galerkinEP3}
    \end{align}
with the initial conditions $\rho^n_0 = \rho_0$ and $u^n_0 = \sum_{i = 1}^n (u_0; \omega_i)_{L^2(\T^d)}\omega_i$. Let us note that the existence of such Galerkin approximations is a standard argument connecting the method of characteristics to solve the continuity equation \eqref{galerkingEP1} for fixed $u^n$ and Schauder's fixed point theorem to obtain the solution. Note that 
\begin{align}
    u^n\in C^1([0, T); W^{6, 2}_0(\T^d)).
\end{align}
In particular the density $\rho^n$ satisfies the equation
\begin{align}\label{formuladensitygalerkin}
    \rho^n(t, x) = \rho_0(X^n(0; t, x))\exp\left(-\frac{1}{\eps}\int_0^t \divergence_x(u^n(\tau, X^n(\tau; t, x)))\diff\tau\right),
\end{align}
where $X^n$ is the flow map defined with
\begin{align}\label{velocityflowsystem}
    \left\{\begin{array}{ll}
    \frac{\p}{\p\tau}X^n(\tau; t, x) = \frac{1}{\eps}u^n(\tau, X^n(\tau; t, x))\\
    X^n(t; t, x) = x.
    \end{array}\right.
\end{align}
Now we want to derive some sort of compactness for the solutions of Galerkin system \eqref{galerkingEP1} - \eqref{galerkinEP3}.
\begin{lem}\label{lem:compactnessforgalerkin}
Let $(\rho^n, u^n, \Phi_{\rho^n})$ be the solution of the Galerkin system \eqref{galerkingEP1} - \eqref{galerkinEP3}. Then 
\begin{enumerate}[label = (B\arabic*)]
    \item $\{\rho^n\}_{n\in\N}$ is bounded in $L^\infty((0, T)\times\T^d)$ \label{galerkincompactness1},
    \item $\{\p_t \rho^n\}_{n\in\N}$ is bounded in $L^2((0, T)\times\T^d)$,
    \item $\{u^n\}_{n\in\N}$ is bounded in $L^\infty((0, T); W^{6, 2}_0(\T^d)))$,
    \item $\{\p_t u^n\}_{n\in\N}$ is bounded in $L^2((0, T)\times\T^d)$,
    \item $\{\nabla_x\Phi_{\rho^n}\}_{n\in\N}$ is bounded in $L^\infty((0, T); L^2(\T^d))$\label{galerkincompactness5}.
\end{enumerate}
In particular, up to the subsequence 
\begin{equation}\label{weakconvergencegalerkin}
    \begin{aligned}
    \rho^n &\overset{*}{\rightharpoonup} \rho \quad\qquad\text{ weakly$-\ast$ in $L^\infty((0, T)\times\T^d)$},\\
    \p_t \rho^n &\rightharpoonup\p_t\rho \qquad\text{ weakly in $L^2((0, T)\times \T^d)$},\\
    \p_t u^n &\rightharpoonup \p_t u \qquad\text{ weakly in $L^2((0, T)\times\T^d)$},\\
    u^n &\overset{*}{\rightharpoonup} u \quad\qquad\text{ weakly$-\ast$ in $L^\infty((0, T); W^{6, 2}_0(\T^d))$},
    \end{aligned}
\end{equation}
furthermore 
\begin{equation}\label{strongconvergencesgalerkin}
    \begin{aligned}
    \rho^n &\rightarrow \rho &\text{ strongly in $L^2((0, T); W^{2, 2}_0(\T^d))$},\\
    u^n &\rightarrow u &\text{ strongly in $L^2((0, T); W^{1,2}_0(\T^d))$},
    \end{aligned}
\end{equation}
and 
\begin{align}\label{potentialconvergencegalerkin}
    \nabla_x\Phi_{\rho^n} \rightarrow \nabla_x\Phi_{\rho} \quad\text{ strongly in $L^2((0, T)\times\T^d)$}.
\end{align}
\end{lem}
\begin{proof}
After multiplying \eqref{galerkinEP2} by $c_i$ and summing over $i$ we obtain the following 
\begin{equation}\label{almostenergygalerkin}
    \begin{aligned}
        &\frac{d}{dt}\int_{\T^d}\frac{1}{2}\rho^n |u^n|^2\diff x + \int_{\T^d}\frac{b}{\eps}\divergence_x \left[\nabla_x\Phi_{\rho^n}\otimes \nabla_x\Phi_{\rho^n} - \frac{1}{2}|\nabla_x\Phi_{\rho^n}|^2\mathbb{I}_d\right] \cdot u^n\diff x\\
        & + \int_{\T^d}\frac{c}{\eps}\divergence_x\left[\frac{1}{2}|\nabla_x\rho^n|^2\mathbb{I}_d + \nabla_x\rho^n\otimes\nabla_x\rho^n \right] \cdot u^n\diff x - \frac{c}{\eps}\int_{\T^d}\rho^n\nabla_x\rho^n\cdot\nabla_x(\divergence_x u^n)\diff x\\
        &+ \frac{1}{\eps^2}\int_{\T^d}\rho^n |u^n|^2\diff x + \frac{a}{\eps}\int_{\T^d}\nabla_x((\rho^n)^\gamma) u^n\diff x - \frac{b}{\eps}M_{\rho^n}\int_{\T^d}\nabla_x\Phi_{\rho^n}u^n\diff x + \mu ((u^n; u^n)) = 0,
    \end{aligned}
    \end{equation}
where we used \eqref{galerkingEP1} to write 
\begin{align*}
    \frac{1}{\eps}\int_{\T^d}\rho^n u^n\nabla_x u^n\cdot u^n\diff x = \int_{\T^d}\p_t\rho^n\frac{1}{2}|u^n|^2\diff x.
\end{align*}
Notice that by the virtue of \eqref{identityforPhi} and \eqref{galerkingEP1} 
\begin{align*}
  &\int_{\T^d}\frac{b}{\eps}\divergence_x \left[\nabla_x\Phi_{\rho^n}\otimes \nabla_x\Phi_{\rho^n} - \frac{1}{2}|\nabla_x\Phi_{\rho^n}|^2\mathbb{I}_d\right] \cdot u^n\diff x - \frac{b}{\eps}M_{\rho^n}\int_{\T^d}\nabla_x\Phi_{\rho^n}u^n\diff x\\
  &= -\frac{b}{\eps}\int_{\T^d}\rho^n\nabla_x\Phi_{\rho^n}u^n\diff x = -\frac{d}{dt}\int_{\T^d}\frac{b}{2}|\nabla_x \Phi_{\rho^n}|^2\diff x.
\end{align*}
Similarly by \eqref{identityforgradlaplace} and \eqref{galerkingEP1} 
\begin{align*}
    &\int_{\T^d}\frac{c}{\eps}\divergence_x\left[\frac{1}{2}|\nabla_x\rho^n|^2\mathbb{I}_d + \nabla_x\rho^n\otimes\nabla_x\rho^n \right] \cdot u^n\diff x - \frac{c}{\eps}\int_{\T^d}\rho^n\nabla_x\rho^n\cdot\nabla_x(\divergence_x u^n)\diff x\\
    &= -\frac{c}{\eps}\int_{\T^d}\divergence_x\left[\frac{1}{2}|\nabla_x\rho^n|^2\mathbb{I}_d + \rho^n\Delta\rho^n\mathbb{I}_d - \nabla_x\rho^n\otimes\nabla_x\rho^n\right]\cdot u^n\diff x = -\frac{c}{\eps}\int_{\T^d}\rho^n\nabla_x(\Delta\rho^n)\cdot u^n\diff x \\
    &= -c\int_{\T^d}\Delta\rho^n\p_t\rho^n\diff x = \frac{d}{dt}\int_{\T^d}\frac{c}{2}|\nabla_x\rho^n|^2\diff x.
\end{align*}
Again using \eqref{galerkingEP1} we obtain 
\begin{align*}
    &\frac{a}{\eps}\int_{\T^d}\nabla_x(\rho^n)^\gamma\cdot u^n\diff x = \frac{a}{\eps}\int_{\T^d}\gamma(\rho^n)^{\gamma - 1}\nabla_x\rho^n\cdot u^n\diff x = \frac{a}{\eps}\int_{\T^d}\nabla_x\left(\frac{\gamma}{\gamma - 1}\rho^n\right)^{\gamma -1}\rho^n u^n\diff x\\
    &= a\int_{\T^d}\left(\frac{\gamma}{\gamma - 1}\rho^n\right)^{\gamma -1}\p_t\rho^n\diff x = \frac{d}{dt}\int_{\T^d}\frac{a}{\gamma - 1}(\rho^n)^{\gamma}\diff x.
\end{align*}
Therefore, \eqref{almostenergygalerkin} takes the form 
\begin{align*}
    &\frac{d}{dt}\int_{\T^d}\frac{1}{2}\rho^n |u^n|^2\diff x - \frac{d}{dt}\int_{\T^d}\frac{b}{2}|\nabla_x \Phi_{\rho^n}|^2\diff x\\
    &+ \frac{d}{dt}\int_{\T^d}\frac{c}{2}|\nabla_x\rho^n|^2\diff x + \frac{d}{dt}\int_{\T^d}\frac{a}{\gamma - 1}(\rho^n)^{\gamma}\diff x + \mu((u^n; u^n)) = -\frac{1}{\eps^2}\int_{\T^d}\rho^n|u^n|^2\diff x,
\end{align*}
which we may write as 
\begin{equation}\label{energygalerking}
    \begin{aligned}
    &\int_{\T^d}\left[\frac{1}{2}\rho^n |u^n|^2 -\frac{b}{2}|\nabla_x \Phi_{\rho^n}|^2 + \frac{c}{2}|\nabla_x\rho^n|^2 + \frac{a}{\gamma - 1}(\rho^n)^{\gamma}\right](\tau, x)\diff x + \int_0^\tau\mu((u^n; u^n))\\ &= \int_{\T^d}\frac{1}{2}\rho_0 |u_0|^2 -\frac{b}{2}|\nabla_x \Phi_{\rho^n}|^2(0,x) + \frac{c}{2}|\nabla_x\rho^n|^2(0, x) + \frac{a}{\gamma - 1}(\rho_0)^{\gamma}\diff x -\frac{1}{\eps^2}\int_0^\tau\int_{\T^d}\rho^n|u^n|^2\diff x.
    \end{aligned}
\end{equation}
From this point forward, the choice of the sign of $b$ becomes important.\\
\underline{\textbf{If $\mathbf{b \leq 0}$}}, \eqref{energygalerking} implies that 
\begin{align*}
    \int_0^\tau \|u^n\|_{W^{6,2}_0(\T^d)}^2\diff x \leq \frac{C}{\mu}.
\end{align*}
By the Sobolev embeddings 
\begin{align}\label{morreyforsobolevflow}
W^{6, 2}_0(\T^d) \hookrightarrow C^4(\T^d),
\end{align}
which means that we may deduce
\begin{align*}
    \int_0^\tau \|u^n\|_{C^4(\T^d)}^2\diff x \leq \frac{C}{\mu}.
\end{align*}
With the use of the inequality above and the formula \eqref{formuladensitygalerkin} one gets 
\begin{multline}\label{galerkinboundsdensity}
    \|\rho^n\|_{L^\infty((0, T)\times \T^d)} + \int_0^T \|\p_t\rho^n\|_{L^2(\T^d)}^2 \diff t\\
    + \int_0^T \|\nabla_x\rho^n\|_{L^2(\T^d)}^2 + \|D^2_x\rho^n\|_{L^2(\T^d)}^2 + \|D^3_x\rho^n\|_{L^2(\T^d)}^2 \diff t \leq C(\mu^{-1})
\end{multline}
and after multiplying \eqref{galerkinEP2} by $\p_t c_i$ and summing over all $i$ 
\begin{align}\label{galerkingboundsvelocity}
    \int_0^T\|\p_t u^n\|_{L^2(\T^d)}^2\diff t + \mu\|u^n\|_{L^\infty((0, T); W^{6, 2}_0(\T^d))} \leq C(\mu^{-1}).
\end{align}
That is obtained using the bound
$$
\rho_*\int_{\T^d}|\p_t u^n|^2\diff x \leq \int_{\T^d}\rho^n|\p_t u^n|^2\diff x
$$
and using the Young's inequality together with \eqref{galerkinboundsdensity} to estimate the rest of the terms.\\
\underline{$\mathbf{\frac{2}{K} > b > 0}$ \textbf{and} $\mathbf{\gamma > 2 - \frac{2}{d}}$}. Here, one may obtain similarly as in the proof of Lemma \ref{lem:boundfornegativebees} 
\begin{align}\label{galerkingpotentialinequality}
    \int_{\T^d}|\nabla_x{\Phi_{\rho^n}}|^2\diff x \leq K\int_{\T^d}h(\rho^n)\diff x.
\end{align}
Using that in \eqref{energygalerking} and remembering that $\frac{b}{2} < \frac{1}{K}$ we proceed as in the case $b < 0$.\\
\\
\noindent This ends the proof of the bounds \ref{galerkincompactness1} - \ref{galerkincompactness5}, for both cases. Equipped with the proven bounds, we can easily see \eqref{weakconvergencegalerkin} as a consequence of Banach--Alaoglu Theorem, \eqref{strongconvergencesgalerkin} of Aubin--Lions Lemma and \eqref{potentialconvergencegalerkin} of Poisson's equation \eqref{galerkinEP3}.
\end{proof}

\subsection{Passage to dissipative measure-valued solutions}
Now we may proceed to obtaining the dissipative measure-valued solutions. For this short time, we come back to the index $\mu$, that we have abandoned before, at the beginning of the previous subsection. Notice, that after converging with $n\to\infty$ we obtain from \eqref{energygalerking}
\begin{equation}\label{energyepsilon}
    \begin{aligned}
    &\int_{\T^d}\left[\frac{1}{2}\rho^\mu |u^\mu|^2 -\frac{b}{2}|\nabla_x \Phi_{\rho^\mu}|^2 + \frac{c}{2}|\nabla_x\rho^\mu|^2 + \frac{a}{\gamma - 1}(\rho^\mu)^{\gamma}\right](\tau, x)\diff x\\ &\leq \int_{\T^d}\frac{1}{2}\rho^\mu_0 |u^\mu_0|^2 -\frac{b}{2}|\nabla_x \Phi_{\rho^\mu}|^2(0,x) + \frac{c}{2}|\nabla_x\rho^\mu|^2(0, x) + \frac{a}{\gamma - 1}(\rho^\mu_0)^{\gamma}\diff x - \frac{1}{\eps^2}\int_0^\tau\int_{\T^d}\rho^\mu|u^\mu|^2\diff x
    \end{aligned}
\end{equation}
We use this inequality to converge in \eqref{approximateEKP1} - \eqref{approximateEKP3}. Notice that \eqref{energyepsilon} implies, that $\{\nabla_x\Phi_{\rho^\mu}\}_{\mu > 0}$ is bounded in $L^\infty((0, T); L^2(\T^d))$. This bound is obvious for $b< 0$. In the second case, we argue with the analogous inequality to \eqref{galerkingpotentialinequality}. Similarly, the density $\{\rho^\mu\}_{\mu > 0}$ is bounded in $L^\infty((0, T); L^\gamma(\T^d))$. Therefore, no concentrations appear when converging with the terms above, meaning Theorem \ref{app:fundamental_thm} applies. Notice further, that if $c > 0$, then by Poincar\'{e}'s inequality  $\rho\in L^\infty((0, T); L^2(\T^d))$, thus making sense of the term $\rho\nabla_x\rho$. For the rest of the terms appearing in \eqref{approximateEKP1} - \eqref{approximateEKP3}, we merely have $L^\infty((0, T); L^1(\T^d))$ bounds, and therefore we need to use Lemma \ref{app:concentration}. We lack a straightforward bound only for the term $\{\rho^\mu u^\mu\}_{\mu > 0}$. To obtain it, we notice that the $L^\infty((0, T); L^\gamma(\T^d))$ bound on the density, implies 
\begin{align*}
    \int_{\T^d}\rho^\mu u^\mu\diff x \leq \frac{1}{2}\int_{\T^d}\rho^\mu\diff x + \frac{1}{2}\int_{\T^d}\rho^\mu|u^\mu|^2\diff x \leq C,
\end{align*}
which is again enough to use Lemma \ref{app:concentration} for the convergence. Considerations above give us the weak formulations \eqref{eq:massmvs} - \eqref{eq:poissonmvs}. To obtain \eqref{eq:energyinequalitymvs} we simply apply Lemma \ref{app:concentration} to \eqref{energyepsilon}.

\section{From EKP to CHKS.}

\noindent In the present section we shall finally consider the passage from the Euler--Korteweg--Poisson system to the Cahn--Hillard--Keller--Segel model in the high friction and long time limit. In fact, we will consider the case of dissipative measure-valued solutions for the Euler--Korteweg--Poisson system, for which we proved the existence. But first, we need a tool in the form of the relative entropy inequality, stated in Theorem \ref{prop:relentropyinequality}, only then after we will give a proof of Theorem~\ref{theo:poissonlimit}. As mentioned in the introduction, the proof follows the similar lines to \cite{lattanzio2017fromgas}. When it will be obvious that we talk about the solutions concerning the high friction \eqref{eq:EKP} we will omit the index $\eps$.

\begin{prop}
Let $(r, \Phi_r)$ be a strong solution to \eqref{eq:CHKS} (see Definition \ref{def:strongsolutions}), then $(r, U, \Phi_r)$, where
\begin{align}\label{strongvelocitydefinition}
U = -\eps\nabla_x\left(\frac{a\gamma}{\gamma - 1}r^{\gamma -1} - b\Phi_r - c\Delta_x r\right)
\end{align}
is a strong solution to the system
\begin{equation}\label{strongeuler-poisson}
   \begin{aligned}
		\partial_tr+\frac{1}{\eps}\operatorname{div}_x(r U)&=0,\\
		\partial_t(r U)+ \frac{1}{\eps}\operatorname{div}_x(r U\otimes U)+ \frac{1}{\eps^2}r U&=-\frac{a}{\eps}\nabla_xr^{\gamma} - \frac{b}{\eps}\operatorname{div}_x\left(\frac{1}{2}|\nabla_x\Phi_r|^2\mathbb{I}_d-\nabla_x\Phi_r\otimes\nabla_x\Phi_r \right)\\
		+ \frac{b}{\eps}M_{r}\nabla_x\Phi_{r}+ \frac{c}{\eps}\operatorname{div}_x&\left(\frac{1}{2}|\nabla_xr|^2\mathbb{I}_d + r\Delta_xr\mathbb{I}_d-\nabla_xr\otimes\nabla_xr \right) + e(r, U),\\
		-\Delta_x\Phi_{r}&=r-M_{r},
	\end{aligned}
\end{equation}
for
\begin{align}\label{def:errorterm}
e = e(r, U) := \partial_t(r U)+ \frac{1}{\eps}\operatorname{div}_x(r U\otimes U)
\end{align}
as well as satisfies the energy inequality 
\begin{multline}\label{strongenergyproposition}
\int_{\T^d}^{}\left(\frac{1}{2}r|U|^2+\frac{a}{\gamma-1}r^{\gamma}-\frac{b}{2}|\nabla\Phi_{r}|^2+\frac{c}{2}|\nabla r|^2\right)(\tau, x)\,\diff x \leq \int_0^\tau\int_{\T^d}U\cdot e\diff x\diff t\\
+ \int_{\T^d}^{}\left(\frac{1}{2}r|U|^2+\frac{a}{\gamma-1}r^{\gamma}-\frac{b}{2}|\nabla\Phi_{r}|^2+\frac{c}{2}|\nabla r|^2\right)(0, x)\,\diff x -\frac{1}{\eps^2}\int_0^\tau\int_{\T^d}r|U|^2\diff x\diff t 
\end{multline}
\end{prop}
\begin{proof}
The fact that $U$ solves the system \eqref{strongeuler-poisson} is a straightforward calculation. The energy inequality \eqref{strongenergyproposition} is obtained in the same way as the energy inequality for Galerkin system above, thus we omit the details.
\end{proof}

\subsection{Relative entropy inequality}

Let us first show the following relative entropy inequality.
\begin{thm}\label{prop:relentropyinequality}
	Let $(\nu^{\eps}, m^{\eps})$ be a high friction sequence of dissipative measure-valued solutions of ~\eqref{eq:EKP} with initial data $(\rho^{\eps}_0, u^{\eps}_0)$ and $\gamma \geq 2$, let $(r,\Phi_r)$ be a strong solution (see Definition \ref{def:strongsolutions}) of ~\eqref{eq:CHKS} with initial data $r_0\in C^3(\T^d)$ satisfying $r_0 \geq \rho_*$ for $\rho_* > 0$. Then, the corresponding relative entropy obeys the inequality
	\begin{align*}
		\mathcal{E}^{\eps}_{\operatorname{rel}}(\tau)+\frac{1}{2\eps^2}\int_{0}^{\tau}\int_{\T^d}^{}\overline{\rho^{\eps}|u^{\eps}-U^{\eps}|^2}\,\diff x\diff t\leq \mathcal{E}^{\eps}_{\operatorname{rel}}(0)+C\int_{0}^{t}\mathcal{E}^{\eps}_{\operatorname{rel}}(t)\diff t+O(\eps^4)
	\end{align*}
	for all $\tau\in [0,T]$, where $C$ is some constant depending only on the strong solution $(r,\Phi_r)$.
\end{thm}
\begin{proof}
	Fix $\tau\in [0,T]$. We define the energy of the measure-valued solution $(\nu^{\eps},\mu^{\eps})$ at time $\tau$ as
	\begin{align*}
		\mathcal{E}^{\mathrm{mv}}(\tau):=\int_{\T^d}^{}\frac{1}{2}\overline{\rho^{\eps}|u^{\eps}|^2}+\frac{a}{\gamma-1}\overline{(\rho^{\eps})^{\gamma}}-\frac{b}{2}\overline{|\nabla\Phi_{\rho^\eps}|^2}+\frac{c}{2}\overline{|\nabla\rho^{\eps}|^2}\,\diff x,
	\end{align*}
	and an analogous quantity for a strong solution \eqref{strongeuler-poisson}
	\begin{align*}
	    \mathcal{E}^{\mathrm{s}}(\tau):=\int_{\T^d}^{}\frac{1}{2}r|U|^2+\frac{a}{\gamma-1}r^{\gamma}-\frac{b}{2}|\nabla\Phi_{r}|^2+\frac{c}{2}|\nabla r|^2\,\diff x.
	\end{align*}
	By the energy inequality \eqref{eq:energyinequalitymvs} we obtain
	\begin{align}\label{ineq:ineqmvenergy}
	    \mathcal{E}^{\mathrm{mv}}(\tau) - \mathcal{E}^{\mathrm{mv}}(0) \leq -\frac{1}{\eps^2}\int_0^\tau\int_{\T^d}\overline{\rho|u|^2}\diff x\diff t
	\end{align}
	And with the use of \eqref{strongenergyproposition}
	\begin{align}\label{ineq:ineqstrongenergy}
	    \mathcal{E}^{\mathrm{s}}(\tau) - \mathcal{E}^{\mathrm{s}}(0) \leq -\frac{1}{\eps^2}\int_0^\tau\int_{\T^d}r|U|^2\diff x\diff t + \int_0^\tau\int_{\T^d}U\cdot e\diff x\diff t
	\end{align}
	Keeping that in mind we turn our attention to the weak formulations of continuity equation as well as momentum equation. From the first one:
	\begin{align}\label{eq:relativeconti}
	    -\int_0^\infty \int_{\T^d}\p_t\psi(\overline{\rho} - r) + \frac{1}{\eps}\nabla_x\psi(\overline{\rho u} - rU)\diff x\diff t - \int_{\T^d}\psi(0, x)(\overline{\rho} - r)(0, x)\diff x = 0
	\end{align}
	and from the second one 
	\begin{equation}\label{eq:relativemoment}
	\begin{split}
	    &-\int_0^\infty\int_{\T^d}\left[\p_t\phi(\overline{\rho u} - rU) + \frac{1}{\eps}(\overline{\rho u\otimes u} - rU\otimes U):\nabla_x\phi + \frac{1}{\eps}\DIV_x\phi(\overline{\rho^\gamma} - r^\gamma)\right]\diff x\diff t\\
	    &- \int_{\T^d}\phi(0, x)(\overline{\rho u} - rU)(0, x)\diff x = -\frac{1}{\eps^2}\int_0^\infty\int_{\T^d}\phi(\overline{\rho u} - rU)\diff x\diff t - \int_0^\infty \int_{\T^d}\phi\cdot e\diff x\diff t\\
	    &+\frac{1}{\eps}\int_0^\infty\int_{\T^d}\left[b\,\phi(M_{\overline{\rho}}\,\overline{\nabla_x\Phi_{\overline{\rho}}} - M_r\,\nabla_x\Phi_r) + b\nabla_x\phi:\left(\overline{\nabla_x\Phi_{\overline{\rho}}\otimes \nabla_x\Phi_{\overline{\rho}}} - \nabla_x\Phi_r\otimes \nabla_x\Phi_r\right)\right]\diff x\diff t\\
	    &-\frac{1}{\eps}\int_0^\infty\int_{\T^d}\frac{b}{2}\DIV_x\phi\left(\overline{|\nabla_x\Phi_{\overline{\rho}}|^2} - |\nabla_x\Phi_r|^2\right)\diff x\diff t+\frac{1}{\eps}\int_0^\infty\int_{\T^d}\frac{c}{2}\,\DIV_x\phi\left(\overline{|\nabla_x\rho|^2}-|\nabla_x r|^2\right)\diff x\diff t\\
	    & + \frac{1}{\eps}\int_0^\infty\int_{\T^d}\left[c\,\nabla_x(\DIV_x\phi)(\overline{\rho\nabla_x\rho} - r\nabla_x r) + c\,\nabla_x\phi :(\overline{\nabla_x\rho\otimes\nabla_x\rho} - \nabla_x r\otimes \nabla_x r)\right]\diff x\diff t.
	    \end{split}
	\end{equation}
	Now set 
	\begin{align*}
	    \theta(t)=\left\{\begin{array}{ll}
	          1, &\text{ when }0\leq t < \tau\\
	          \frac{\tau - t}{\mu} + 1, &\text{ when }\tau \leq t < \tau + \mu\\
	          0, &\text{ otherwise}
	    \end{array}\right.,
	\end{align*}
	as well as
	$$
	\psi(t) = \theta(t)\left(\frac{a\gamma}{\gamma - 1}r^{\gamma - 1} - b\,\Phi_r - c\,\Delta_x r - \frac{1}{2}|U|^2\right),\qquad \phi(t) = \theta(t)\,U.
	$$
	Putting it into \eqref{eq:relativeconti} and \eqref{eq:relativemoment}, converging with $\mu\to 0$, we obtain respectively
	\begin{equation}\label{eq:relativeconti2}
	\begin{split}
	    &\int_{\T^d}\left(\frac{a\gamma}{\gamma - 1}r^{\gamma - 1} - b\,\Phi_r - c\,\Delta_x r - \frac{1}{2}|U|^2\right)(\overline{\rho} - r)\Biggr|_0^\tau\diff x\\
	    &-\int_0^\tau \int_{\T^d}\p_t\left(\frac{a\gamma}{\gamma - 1}r^{\gamma - 1} - b\,\Phi_r - c\,\Delta_x r - \frac{1}{2}|U|^2\right)(\overline{\rho} - r)\diff x\diff t\\
	    &- \frac{1}{\eps}\int_0^\tau \int_{\T^d}\nabla_x\left(\frac{a\gamma}{\gamma - 1}r^{\gamma - 1} - b\,\Phi_r - c\,\Delta_x r - \frac{1}{2}|U|^2\right)(\overline{\rho u} - rU)\diff x\diff t = 0,
	\end{split}
	\end{equation}
	and
	\begin{equation}\label{eq:relativemoment2}
	\begin{split}
	    &-\int_0^\tau\int_{\T^d}\left[\p_t U(\overline{\rho u} - rU) + \frac{1}{\eps}(\overline{\rho u\otimes u} - rU\otimes U):\nabla_x U + \frac{1}{\eps}\DIV_x U(\overline{\rho^\gamma} - r^\gamma)\right]\diff x\diff t\\
	    & + \int_{\T^d}U(\overline{\rho u} - rU)\Biggr|_0^\tau\diff x = -\frac{1}{\eps^2}\int_0^\tau\int_{\T^d}U(\overline{\rho u} - rU)\diff x\diff t - \int_0^\tau \int_{\T^d}U\cdot e\diff x\diff t\\
	    &+\frac{1}{\eps}\int_0^\tau\int_{\T^d}\left[b\,U(M_{\overline{\rho}}\,\overline{\nabla_x\Phi_{\overline{\rho}}} - M_r\,\nabla_x\Phi_r) + b\nabla_xU:\left(\overline{\nabla_x\Phi_{\overline{\rho}}\otimes \nabla_x\Phi_{\overline{\rho}}} - \nabla_x\Phi_r\otimes \nabla_x\Phi_r\right)\right]\diff x\diff t\\
	    &-\frac{1}{\eps}\int_0^\tau\int_{\T^d}\frac{b}{2}\DIV_xU\left(\overline{|\nabla_x\Phi_{\overline{\rho}}|^2} - |\nabla_x\Phi_r|^2\right)\diff x\diff t+\frac{1}{\eps}\int_0^\tau\int_{\T^d}\frac{c}{2}\,\DIV_xU(\overline{|\nabla_x\rho|^2}-|\nabla_x r|^2)\diff x\diff t\\
	    & + \frac{1}{\eps}\int_0^\tau\int_{\T^d}\left[c\,\nabla_x(\DIV_xU)(\overline{\rho\nabla_x\rho} - r\nabla_x r) + c\,\nabla_xU :(\overline{\nabla_x\rho\otimes\nabla_x\rho} - \nabla_x r\otimes \nabla_x r)\right]\diff x\diff t.
	    \end{split}
	\end{equation}
	Subtracting  \eqref{ineq:ineqstrongenergy} + \eqref{eq:relativeconti2} + \eqref{eq:relativemoment2} from \eqref{ineq:ineqmvenergy}, we get
	\begin{equation}\label{ineq:relative1}
	    \begin{split}
	        & \mathcal{E}_{\mathrm{rel}}(\tau) - \mathcal{E}_{\mathrm{rel}}(0) \leq -\frac{1}{\eps^2}\int_0^\tau\int_{\T^d}\left(\overline{\rho|u|^2} - r|U|^2 - U(\overline{\rho u} - rU)\right)\diff x\diff t\\
	        &-\int_0^\tau\int_{\T^d}\left[\p_t U(\overline{\rho u} - rU) + \p_t\left(\frac{a\gamma}{\gamma - 1}r^{\gamma - 1} - b\,\Phi_r - c\,\Delta_x r - \frac{1}{2}|U|^2\right)(\overline{\rho} - r) \right]\diff x\diff t\\
	        &- \frac{1}{\eps}\int_0^\tau \int_{\T^d}\nabla_x\left(\frac{a\gamma}{\gamma - 1}r^{\gamma - 1} - b\,\Phi_r - c\,\Delta_x r - \frac{1}{2}|U|^2\right)(\overline{\rho u} - rU)\diff x\diff t\\
	        &-\frac{1}{\eps}\int_0^\tau\int_{\T^d}\left[b\,U(M_{\overline{\rho}}\,\overline{\nabla_x\Phi_{\overline{\rho}}} - M_r\,\nabla_x\Phi_r) + b\nabla_xU:\left(\overline{\nabla_x\Phi_{\overline{\rho}}\otimes \nabla_x\Phi_{\overline{\rho}}} - \nabla_x\Phi_r\otimes \nabla_x\Phi_r\right)\right]\diff x\diff t\\
	        &+\frac{1}{\eps}\int_0^\tau\int_{\T^d}\frac{b}{2}\DIV_xU\left(\overline{|\nabla_x\Phi_{\overline{\rho}}|^2} - |\nabla_x\Phi_r|^2\right)\diff x\diff t-\frac{1}{\eps}\int_0^\tau\int_{\T^d}\frac{c}{2}\,\DIV_xU(\overline{|\nabla_x\rho|^2}-|\nabla_x r|^2)\diff x\diff t\\
	        & - \frac{1}{\eps}\int_0^\tau\int_{\T^d}\left[c\,\nabla_x(\DIV_xU)(\overline{\rho\nabla_x\rho} - r\nabla_x r) + c\,\nabla_xU :(\overline{\nabla_x\rho\otimes\nabla_x\rho} - \nabla_x r\otimes \nabla_x r)\right]\diff x\diff t\\
	        & - \frac{1}{\eps}\int_0^\tau\int_{\T^d}\left[(\overline{\rho u\otimes u} - rU\otimes U):\nabla_x U + \frac{1}{\eps}\DIV_x U(\overline{\rho^\gamma} - r^\gamma)\right]\diff x\diff t.
	    \end{split}
	\end{equation}
	To estimate the right-hand side of the inequality above first notice that by the definition \eqref{strongvelocitydefinition} the function $U$ solves the equation
	$$
	\p_t U + \frac{1}{\eps}\nabla_x U\,U = -\frac{1}{\eps^2}U - \frac{1}{\eps}\nabla_x\left(\frac{a\gamma}{\gamma - 1}r^{\gamma - 1} - b\,\Phi_r - c\,\Delta_x r\right) + \frac{e}{r},
	$$
	where $e$ is defined by \eqref{def:errorterm}. Hence
	\begin{align*}
	    -U(\overline{\rho} - r)\left[\p_t U + \frac{1}{\eps}\nabla_x U\,U\right] = -U(\overline{\rho} - r)\left[-\frac{1}{\eps^2}U - \frac{1}{\eps}\nabla_x\left(\frac{a\gamma}{\gamma - 1}r^{\gamma - 1} - b\,\Phi_r - c\,\Delta_x r\right) + \frac{e}{r}\right]
	\end{align*}
	or equivalently
	\begin{equation}\label{eq:strongu1}
	    \begin{split}
	    &\p_t\left(-\frac{1}{2}|U|^2\right) - \frac{1}{\eps}\nabla_x U : (U\otimes U)(\overline{\rho} - r) = \frac{1}{\eps^2}|U|^2(\overline{\rho} - r)\\
	    &+ \frac{1}{\eps}U\cdot\nabla_x\left(\frac{a\gamma}{\gamma-1}r^{\gamma-1} - b\,\Phi_r - c\,\Delta_x r\right)(\overline{\rho} - r) - \frac{e}{r}U(\overline{\rho} - r).
	    \end{split}
	\end{equation}
    In a similar fashion, we may obtain
    \begin{equation}\label{eq:strongu2}
        \begin{split}
            &\p_t U(\overline{\rho u} - rU) + \frac{1}{\eps}\nabla_x U : (\overline{\rho u}\otimes U - r\,U\otimes U) = -\frac{1}{\eps^2}U\,\overline{\rho u} + \frac{1}{\eps^2}r|U|^2\\
            &- \frac{1}{\eps}(\overline{\rho u} - rU)\cdot\nabla_x\left(\frac{a\gamma}{\gamma - 1}r^{\gamma - 1} - b\,\Phi_r - c\,\Delta_x r\right) + \frac{e}{r}(\overline{\rho u} - rU).
        \end{split}
    \end{equation}
    Combining \eqref{eq:strongu1} and \eqref{eq:strongu2} we get
    \begin{equation}\label{eq:reduction1}
        \begin{split}
            &\p_t\left(-\frac{1}{2}|U|^2\right) + \p_t U(\overline{\rho u} - rU) + \frac{1}{\eps}\nabla_x\left(-\frac{1}{2}|U|^2\right)(\overline{\rho u} - rU) + \frac{1}{\eps}\nabla_x U: (\overline{\rho u\otimes u} - r\,U\otimes U)\\
            &- \frac{1}{\eps}\nabla_x U:(U\otimes\overline{\rho u} - \overline{\rho u} \otimes U) = -\frac{1}{\eps^2}(\overline{\rho u}\cdot U - \overline{\rho}|U|^2)\\
            &- \frac{1}{\eps}\nabla_x\left(\frac{a\gamma}{\gamma - 1}r^{\gamma - 1} - b\,\Phi_r - c\,\Delta_x r\right)(\overline{\rho u} - \overline{\rho}U) + \frac{e}{r}(\overline{\rho u} - \overline{\rho}U) + \frac{1}{\eps}\nabla_x U:\overline{\rho(u - U)\otimes(u - U)}.
        \end{split}
    \end{equation}
    To end the calculation notice that since $U$ is a gradient of some function, then $\nabla_x U$ is a symmetric matrix, thus
    $$
    \frac{1}{\eps}\nabla_x U:(U\otimes\overline{\rho u} - \overline{\rho u} \otimes U) = 0.
    $$
    Here, we introduce the notation for relative pressure $p(\rho) = \rho^\gamma$.
    \begin{align}\label{relativepressuredef}
    p(\rho|r) = p(\rho) - p(r) - p'(r)(\rho - r).
    \end{align}
    We may insert \eqref{eq:reduction1} into \eqref{ineq:relative1} to get
    \begin{equation}\label{ineq:relative3}
        \begin{split}
            & \mathcal{E}_{\mathrm{rel}}(\tau) - \mathcal{E}_{\mathrm{rel}}(0) \leq -\frac{1}{\eps^2}\int_0^\tau\overline{\rho|u - U|^2}\diff x\diff t - \int_0^\tau\int_{\T^d}\frac{e}{r}\left(\overline{\rho u} - \overline{\rho U}\right)\diff x\diff t\\
            &- \frac{1}{\eps}\int_0^\tau\int_{\T^d}\DIV_x U \,\overline{p(\rho|r)}\diff x\diff t - \frac{1}{\eps}\int_0^\tau \int_{\T^d}\nabla_x U:\overline{\rho (u - U)\otimes (u-U)}\diff x\diff t\\
            &-\frac{1}{\eps}\int_0^\tau\int_{\T^d}\left[b\,U(M_{\overline{\rho}}\,\overline{\nabla_x\Phi_{\overline{\rho}}} - M_r\,\nabla_x\Phi_r) + b\nabla_xU:\left(\overline{\nabla_x\Phi_{\overline{\rho}}\otimes \nabla_x\Phi_{\overline{\rho}}} - \nabla_x\Phi_r\otimes \nabla_x\Phi_r\right)\right]\diff x\diff t\\
	        &+\frac{1}{\eps}\int_0^\tau\int_{\T^d}\frac{b}{2}\DIV_xU\left(\overline{|\nabla_x\Phi_{\overline{\rho}}|^2} - |\nabla_x\Phi_r|^2\right)\diff x\diff t-\frac{1}{\eps}\int_0^\tau\int_{\T^d}\frac{c}{2}\,\DIV_xU(\overline{|\nabla_x\rho|^2}-|\nabla_x r|^2)\diff x\diff t\\
	        & - \frac{1}{\eps}\int_0^\tau\int_{\T^d}\left[c\,\nabla_x(\DIV_xU)(\overline{\rho\nabla_x\rho} - r\nabla_x r) + c\,\nabla_xU :(\overline{\nabla_x\rho\otimes\nabla_x\rho} - \nabla_x r\otimes \nabla_x r)\right]\diff x\diff t\\
	        &+\int_0^\tau\int_{\T^d}\p_t(b\,\Phi_r + c\Delta_x r)(\overline{\rho} - r)\diff x\diff t + \frac{1}{\eps}\int_0^\tau\int_{\T^d}\nabla_x(b\,\Phi_r + c\,\Delta_x r)U(\overline{\rho} - r)\diff x\diff t.
        \end{split}
    \end{equation}
    Let us first handle the terms related the Euler-Poisson equation. Notice that since $(r, \Phi_r)$ is a strong solution to $\eqref{eq:CHKS}_2$ and $(r, U)$ is a strong solution to the continuity equation, then using the distributional formulation \eqref{eq:poissonmvs} for $(\overline{\rho}, \overline{\nabla_x\Phi_{\overline{\rho}}})$, we get
    \begin{equation}\label{eq:reductionpoisson4}
    \begin{split}
            &\int_0^\tau\int_{\T^d}(r - M_r)\overline{\nabla_x\Phi_{\overline{\rho}}}\cdot U \diff x\diff t + \int_0^\tau\int_{\T^d}(\overline{\rho} - M_{\overline{\rho}})\nabla_x\Phi_r\cdot U\diff x\diff t\\
            &=-\int_0^\tau\int_{\T^d}\Delta_x\Phi_r\,\overline{\nabla_x\Phi_{\overline{\rho}}}\cdot U \diff x\diff t + \int_0^\tau\int_{\T^d}\nabla_x\Phi_{\overline{\rho}}\cdot\nabla_x(\nabla_x\Phi_r\cdot U)\diff x\diff t\\
            &= -\int_0^\tau\int_{\T^d}\Delta_x\Phi_r\,\overline{\nabla_x\Phi_{\overline{\rho}}}\cdot U \diff x\diff t + \int_0^\tau\int_{\T^d}\nabla_x\Phi_r\otimes\nabla_x\Phi_{\overline{\rho}}:\nabla_x U\diff x\diff t\\
            &+ \int_0^\tau\int_{\T^d}\nabla_x\Phi_{\overline{\rho}}\cdot D^2_x\Phi_r\cdot U\diff x\diff t = \int_0^\tau\int_{\T^d}\nabla_x\Phi_{r}\cdot\nabla_x(\nabla_x\Phi_{\overline{\rho}}\cdot U)\diff x\diff t\\
            &+\int_0^\tau\int_{\T^d}\nabla_x\Phi_r\otimes\nabla_x\Phi_{\overline{\rho}}:\nabla_x U\diff x\diff t +
             \int_0^\tau\int_{\T^d}\nabla_x\Phi_{\overline{\rho}}\cdot D^2_x\Phi_r\cdot U\diff x\diff t \\
            &=\int_0^\tau\int_{\T^d}\overline{\nabla_x\Phi_{\overline{\rho}}}\otimes\nabla_x\Phi_r : \nabla_x U\diff x\diff t + \int_0^\tau\int_{\T^d}\nabla_x\Phi_{r}\otimes\overline{\nabla_x\Phi_{\overline{\rho}}} : \nabla_x U\diff x\diff t \\
            & + \int_0^\tau\int_{\T^d}\nabla_x\Phi_{r}\cdot D^2_x\Phi_{\overline{\rho}}\cdot U\diff x\diff t + \int_0^\tau\int_{\T^d}\nabla_x\Phi_{\overline{\rho}}\cdot D^2_x\Phi_r\cdot U\diff x\diff t \\
            &  = -\int_0^\tau\int_{\T^d}\overline{\nabla_x\Phi_{\overline{\rho}}}\cdot\nabla_x\Phi_r\,\DIV_x U\diff x\diff t \\
            &+ \int_0^\tau\int_{\T^d}\overline{\nabla_x\Phi_{\overline{\rho}}}\otimes\nabla_x\Phi_r : \nabla_x U\diff x\diff t + \int_0^\tau\int_{\T^d}\nabla_x\Phi_{r}\otimes\overline{\nabla_x\Phi_{\overline{\rho}}} : \nabla_x U\diff x\diff t.
    \end{split}
    \end{equation}
    From the properties of the Poisson's equation (here, $\Gamma$ denotes the Newtonian potential) and the fact that $M_{\overline{\rho}}$ and $M_r$ are independent of time, we conclude that
    \begin{align*}
        &\int_0^\tau\int_{\T^d}(\Gamma*\p_t r) (\overline{\rho} - r)\diff x\diff t = \int_0^\tau\int_{\T^d}\Gamma* \p_t r \,\left((\overline{\rho} - M_{\overline{\rho}}) - (r - M_r)\right)\diff x\diff t \\
        & + \int_0^\tau\int_{\T^d}\Gamma* \p_t r \,\left(M_{\overline{\rho}} - M_r\right)\diff x\diff t = \int_0^\tau\int_{\T^d}\p_t r\,(\Phi_{\overline{\rho}} - \Phi_r)\diff x\diff t\\
        &+ \int_{\T^d} (\Gamma*r) \,\left( M_{\overline{\rho}} - M_r\right)\diff x\Biggr|_0^\tau -\int_0^\tau\int_{\T^d} (\Gamma*r) \,\p_t\left( M_{\overline{\rho}} - M_r\right)\diff x\diff t\\
        & = \int_0^\tau\int_{\T^d}\p_t r\,(\Phi_{\overline{\rho}} - \Phi_r)\diff x\diff t+ \int_{\T^d} (\Gamma*r) \,\left( M_{\overline{\rho}} - M_r\right)\diff x\Biggr|_0^\tau.
    \end{align*}
    Calculating further 
    \begin{align*}
        &\int_{\T^d} (\Gamma*r) \,\left( M_{\overline{\rho}} - M_r\right)\diff x\Biggr|_0^\tau\\
        &= \left( M_{\overline{\rho}} - M_r\right) \left(\int_{\T^d} (\Gamma*r)(\tau, x) \,\diff x - \int_{\T^d} (\Gamma*r)(0, x) \,\diff x\right)\\
        &= \left( M_{\overline{\rho}} - M_r\right) \left(\int_{\T^d} \int_{x - \T^d}\Gamma(y)\, r(\tau, x - y) \diff y\diff x - \int_{\T^d} \int_{x - \T^d}\Gamma(y)\, r(0, x - y) \diff y\diff x\right)\\
        &= \left( M_{\overline{\rho}} - M_r\right) \left(\int_{\T^d} \int_{y + \T^d}\Gamma(y)\, r(\tau, x - y) \diff x\diff y - \int_{\T^d} \int_{y + \T^d}\Gamma(y)\, r(0, x - y) \diff x\diff y\right)\\
        &= \left( M_{\overline{\rho}} - M_r\right) \left(\int_{\T^d}\Gamma(y)M_r\diff y - \int_{\T^d}\Gamma(y)M_r\diff y\right)\\
        &= 0.
    \end{align*}
    Hence
    \begin{equation}\label{eq:reductionpoisson5}
        \begin{split}
            &-\int_0^\tau\int_{\T^d}\,\p_t\,\Phi_r(\overline{\rho} - r)\diff x\diff t - \frac{1}{\eps}\int_0^\tau\int_{\T^d}\nabla_x\,\Phi_r\cdot U(\overline{\rho} - r)\diff x\diff t\\
            &= -\int_0^\tau\int_{\T^d}(\Gamma*\p_t r) (\overline{\rho} - r)\diff x\diff t - \frac{1}{\eps}\int_0^\tau\int_{\T^d}\nabla_x\,\Phi_r\cdot U(\overline{\rho} - r)\diff x\diff t\\
            &= -\int_0^\tau\int_{\T^d}\p_t r(\Phi_{\overline{\rho}} - \Phi_r)\diff x\diff t - \frac{1}{\eps}\int_0^\tau\int_{\T^d}\nabla_x\,\Phi_r\cdot U(\overline{\rho} - r)\diff x\diff t\\
            & = \frac{1}{\eps}\int_0^\tau\int_{\T^d}\DIV_x(rU)(\Phi_{\overline{\rho}} - \Phi_r)\diff x\diff t - \frac{1}{\eps}\int_0^\tau\int_{\T^d}\nabla_x\,\Phi_r\cdot U(\overline{\rho} - r)\diff x\diff t\\
            & = -\frac{1}{\eps}\int_0^\tau\int_{\T^d}\overline{\rho}\overline{\nabla_x\Phi_{\overline{\rho}}}\cdot r\, U\diff x\diff t - \frac{1}{\eps}\int_0^\tau\int_{\T^d}\overline{\rho}\,U\cdot\nabla_x\,\Phi_r\diff x\diff t\\
            &+ \frac{2}{\eps}\int_0^\tau\int_{\T^d}r\,U\cdot\nabla_x\,\Phi_r\diff x\diff t.
        \end{split}
    \end{equation}
    With the use of both \eqref{eq:reductionpoisson5} and \eqref{eq:reductionpoisson4} and some basic calculations 
    \begin{equation}\label{eq:reductionpoisson3}
    \begin{split}
            &\frac{1}{\eps}\int_0^\tau\int_{\T^d}\left[b\,U(M_{\overline{\rho}}\,\overline{\nabla_x\Phi_{\overline{\rho}}} - M_r\,\nabla_x\Phi_r) + b\nabla_xU:\left(\overline{\nabla_x\Phi_{\overline{\rho}}\otimes \nabla_x\Phi_{\overline{\rho}}} - \nabla_x\Phi_r\otimes \nabla_x\Phi_r\right)\right]\diff x\diff t\\
	        &-\frac{1}{\eps}\int_0^\tau\int_{\T^d}\frac{b}{2}\DIV_xU\left(\overline{|\nabla_x\Phi_{\overline{\rho}}|^2} - |\nabla_x\Phi_r|^2\right)\diff x\diff t -\int_0^\tau\int_{\T^d}b\,\p_t\,\Phi_r(\overline{\rho} - r)\diff x\diff t\\
	        &- \frac{1}{\eps}\int_0^\tau\int_{\T^d}\nabla_x(b\,\Phi_r)\cdot U(\overline{\rho} - r)\diff x\diff t = - \frac{b}{2\eps}\int_0^\tau\int_{\T^d}\overline{\left|\nabla_x\Phi_{\overline{\rho}} - \nabla_x\Phi_r\right|^2}\cdot\DIV_xU\diff x\diff t\\
            &+ \frac{b}{\eps}\int_0^\tau\int_{\T^d}\overline{\left(\nabla_x\Phi_{\overline{\rho}} - \nabla_x\Phi_r\right) \otimes \left(\nabla_x\Phi_{\overline{\rho}} - \nabla_x\Phi_r\right)} : \nabla_x U\diff x\diff t\\
            &+\frac{b}{\eps}\int_0^\tau\int_{\T^d}(M_{\overline{\rho}} - M_r)(\overline{\nabla_x\Phi_{\overline{\rho}}} - \nabla_x\Phi_r)\,U\diff x\diff t.
    \end{split}
    \end{equation}
    % and after converging to measure-valued solutions:
    % \begin{equation}\label{eq:measurevaluedpoissontrick}
    %     \begin{split}
    %         &\int_0^\tau\int_{\T^d}(r - M_r)\overline{\nabla_x \Phi_{\overline{\rho}}}\cdot\phi \diff x\diff t + \int_0^\tau\int_{\T^d}(\overline{\rho} - M_{\overline{\rho}})\nabla_x\Phi_r\cdot\phi\diff x\diff t\\
    %         &= -\int_0^\tau\int_{\T^d}\overline{\nabla_x\Phi_{\overline{\rho}}}\cdot\nabla_x\Phi_r\,\DIV_x\phi\diff x\diff t\\
    %         &+ \int_0^\tau\int_{\T^d}\overline{\nabla_x\Phi_{\overline{\rho}}}\otimes\nabla_x\Phi_r : \nabla_x\phi\diff x\diff t + \int_0^\tau\int_{\T^d}\nabla_x\Phi_{r}\otimes\overline{\nabla_x\Phi_{\overline{\rho}}} : \nabla_x\phi\diff x\diff t 
    %     \end{split}
    % \end{equation}
   \noindent Now we may move into the terms from Euler--Korteweg. As $(r, U)$ is a strong solution to the continuity equation (see \eqref{strongeuler-poisson}), we may perform the following calculations 
   \begin{equation}\label{eq:reductionkorteweg2}
   \begin{aligned}
            &-\int_0^\tau\int_{\T^d}\frac{1}{2}\,\DIV_xU(\overline{|\nabla_x\rho|^2}-|\nabla_x r|^2)\diff x\diff t + \nabla_x(\DIV_xU)(\overline{\rho\nabla_x\rho} - r\nabla_x r)\diff x\diff t\\
            & - \int_0^\tau\int_{\T^d}\nabla_xU :(\overline{\nabla_x\rho\otimes\nabla_x\rho} - \nabla_x r\otimes \nabla_x r)\diff x\diff t =  -\frac{1}{2}\int_0^\tau\int_{\T^d}\overline{|\nabla_x\rho - \nabla_x r|^2}\cdot \DIV_x U\diff x\diff t\\
            &+ \frac{1}{2}\int_0^\tau\int_{\T^d}|\nabla_x r|^2\cdot\DIV_x U\diff x\diff t - \int_0^\tau\int_{\T^d}\overline{\nabla_x\rho}\cdot\nabla_x r\,\DIV_x U\diff x\diff t\\
            &- \int_0^\tau\int_{\T^d}(\overline{\nabla_x\rho - \nabla_x r)\otimes(\nabla_x\rho - \nabla_x r)} : \nabla_x U\diff x\diff t + 2\int_0^\tau\int_{\T^d}\nabla_x r\otimes\nabla_x r : \nabla_x U\diff x\diff t\\
            &- \int_0^\tau\int_{\T^d}\nabla_x r\otimes \overline{\nabla_x\rho} : \nabla_x U\diff x\diff t - \int_0^\tau\int_{\T^d} \overline{\nabla_x\rho} \otimes \nabla_x r: \nabla_x U\diff x\diff t\\
            &-\int_0^\tau\int_{\T^d}\nabla_x(\DIV_xU)(\overline{\rho\nabla_x\rho} - r\nabla_x r)\diff x\diff t.
    \end{aligned}
    \end{equation}
    Similarly, due to the Sobolev regularity (see \eqref{def:sobolevregularitydensity} and Definition \ref{def:strongsolutions}) of the densities, we calculate the next terms 
    \begin{equation}\label{eq:reductionkorteweg3}
    \begin{aligned}
            &\int_0^\tau\int_{\T^d}\p_t(\Delta_x r)(\overline{\rho} - r)\diff x\diff t + \frac{1}{\eps}\int_0^\tau\int_{\T^d}U(\overline{\rho} - r)\nabla_x(\Delta_x r)\diff x\diff t\\
            &= \frac{1}{\eps}\int_0^\tau\int_{\T^d}\nabla_x(\DIV_x(rU))\cdot\nabla_x(\overline{\rho} - r)\diff x\diff t + \frac{1}{\eps}\int_0^\tau\int_{\T^d}U(\overline{\rho} - r)\nabla_x(\Delta_x r)\diff x\diff t\\
            & = \frac{1}{\eps}\int_0^\tau\int_{\T^d}\nabla_x(\DIV_x U)\,r\,\nabla_x(\overline{\rho} - r) + \DIV_x U\,\nabla_x r\cdot\nabla_x(\overline{\rho} - r) + \nabla_x(U\,\nabla_x r)\cdot\nabla_x(\overline{\rho} - r)\diff x\diff t\\
            &+ \frac{1}{\eps}\int_0^\tau\int_{\T^d}\nabla_x(\DIV_x U)\,(\overline{\rho} - r)\,\nabla_x r + \DIV_x U\,\nabla_xr\cdot\nabla_x(\overline{\rho} - r) - (U\,\nabla_x(\overline{\rho} - r))\cdot\Delta_x r\diff x\diff t\\
            & = \frac{1}{\eps}\int_0^\tau\int_{\T^d}\nabla_x(\DIV_x U)(\overline{\rho}\,\nabla_x r + \nabla_x\overline{\rho}\,r - 2\,r\,\nabla_x r)\diff x\diff t + \frac{1}{\eps}\int_0^\tau\int_{\T^d}\nabla_x\overline{\rho}\cdot\nabla_x r\, \DIV_x U\diff x\diff t\\
            &  + \frac{1}{\eps}\int_0^\tau\int_{\T^d}\DIV_x U\,\nabla_x r\cdot\nabla_x(\overline{\rho} - r) + \nabla_x(U\,\nabla_x r)\cdot\nabla_x(\overline{\rho} - r) - (U\,\nabla_x(\overline{\rho} - r))\cdot\Delta_x r \diff x\diff t\\
            &- \frac{1}{\eps}\int_0^\tau\int_{\T^d}\DIV_x U |\nabla_x r|^2\diff x\diff t = - \frac{1}{\eps}\int_0^\tau\int_{\T^d}\DIV_x U |\nabla_x r|^2\diff x\diff t\\
            & +\frac{1}{\eps}\int_0^\tau\int_{\T^d}\nabla_x U : (\overline{\nabla_x\rho}\otimes\nabla_x r + \nabla_x r \otimes\overline{\nabla_x\rho} - 2\nabla_x r \otimes \nabla_x r)\diff x\diff t\\
            & + \frac{1}{\eps}\int_0^\tau\int_{\T^d}\nabla_x(\DIV_x U)(\overline{\rho}\,\nabla_x r + \overline{\nabla_x\rho}\,r - 2\,r\,\nabla_x r)\diff x\diff t + \frac{1}{\eps}\int_0^\tau\int_{\T^d}\overline{\nabla_x\rho}\cdot\nabla_x r\, \DIV_x U\diff x\diff t.
    \end{aligned}
    \end{equation}
    Thus, combining both \eqref{eq:reductionkorteweg2} and \eqref{eq:reductionkorteweg3} 
\begin{equation}\label{eq:reductionkorteweg1}
    \begin{split}
            &-\frac{c}{\eps}\int_0^\tau\int_{\T^d}\frac{1}{2}\,\DIV_xU(\overline{|\nabla_x\rho|^2}-|\nabla_x r|^2)\diff x\diff t + \nabla_x(\DIV_xU)(\overline{\rho\nabla_x\rho} - r\nabla_x r)\diff x\diff t\\
            & - \frac{c}{\eps}\int_0^\tau\int_{\T^d}\nabla_xU :(\overline{\nabla_x\rho\otimes\nabla_x\rho} - \nabla_x r\otimes \nabla_x r)\diff x\diff t + \int_0^\tau\int_{\T^d}\p_t(\Delta_x r)(\overline{\rho} - r)\diff x\diff t\\
            &+ \frac{1}{\eps}\int_0^\tau\int_{\T^d}U(\overline{\rho} - r)\nabla_x(\Delta_x r)\diff x\diff t =-\frac{c}{\eps}\int_0^\tau\int_{\T^d}\nabla_x(\DIV_xU)\overline{(\rho - r)\nabla_x(\rho - r)}\diff x\diff t \\
            &- \frac{c}{\eps}\int_0^\tau\int_{\T^d}(\overline{\nabla_x\rho - \nabla_x r)\otimes(\nabla_x\rho - \nabla_x r)} : \nabla_x U\diff x\diff t-\frac{c}{2\eps}\int_0^\tau\int_{\T^d}\overline{|\nabla_x\rho - \nabla_x r|^2}\cdot \DIV_x U\diff x\diff t.
    \end{split}
\end{equation}
Using both \eqref{eq:reductionpoisson3} and \eqref{eq:reductionkorteweg1} in \eqref{ineq:relative3} we obtain 
\begin{equation}\label{ineq:relative4}
        \begin{split}
            & \mathcal{E}_{\mathrm{rel}}(\tau) - \mathcal{E}_{\mathrm{rel}}(0) \leq -\frac{1}{\eps^2}\int_0^\tau\overline{\rho|u - U|^2}\diff x\diff t - \int_0^\tau\int_{\T^d}\frac{e}{r}\left(\overline{\rho u} - \overline{\rho U}\right)\diff x\diff t\\
            &- \frac{1}{\eps}\int_0^\tau\int_{\T^d}\DIV_x U \,\overline{p(\rho|r)}\diff x\diff t - \frac{1}{\eps}\int_0^\tau \int_{\T^d}\nabla_x U:\overline{\rho (u - U)\otimes (u-U)}\diff x\diff t\\
            &- \frac{b}{\eps}\int_0^\tau\int_{\T^d}\overline{\left(\nabla_x\Phi_{\overline{\rho}} - \nabla_x\Phi_r\right) \otimes \left(\nabla_x\Phi_{\overline{\rho}} - \nabla_x\Phi_r\right)} : \nabla_x U\diff x\diff t\\
            &+ \frac{b}{2\eps}\int_0^\tau\int_{\T^d}\overline{\left|\nabla_x\Phi_{\overline{\rho}} - \nabla_x\Phi_r\right|^2}\cdot\DIV_xU\diff x\diff t-\frac{c}{\eps}\int_0^\tau\int_{\T^d}\nabla_x(\DIV_xU)\overline{(\rho - r)\nabla_x(\rho - r)}\diff x\diff t \\
            &- \frac{c}{\eps}\int_0^\tau\int_{\T^d}(\overline{\nabla_x\rho - \nabla_x r)\otimes(\nabla_x\rho - \nabla_x r)} : \nabla_x U\diff x\diff t-\frac{c}{2\eps}\int_0^\tau\int_{\T^d}\overline{|\nabla_x\rho - \nabla_x r|^2}\cdot \DIV_x U\diff x\diff t\\
            &-\frac{b}{\eps}\int_0^\tau\int_{\T^d}(M_{\overline{\rho}} - M_r)(\overline{\nabla_x\Phi_{\overline{\rho}}} - \nabla_x\Phi_r)\, U\diff x\diff t.
        \end{split}
    \end{equation}
Finally, we have a good enough form of the right-hand side to start estimating the terms separately. Notice that, since $r$ is a strong solution from Definition \ref{def:strongsolutions}, we know that $\frac{1}{\eps}\|U\|_\infty$, $\frac{1}{\eps}\|\nabla_x U\|_\infty$, $\frac{c}{\eps}\|D^2_x U\|_\infty$, all belong to $O(1)$. Thus, the bounds on the terms
\begin{align*}
&\frac{1}{\eps}\int_0^\tau \int_{\T^d}\nabla_x U:\overline{\rho (u - U)\otimes (u-U)}\diff x\diff t, \quad \frac{b}{2\eps}\int_0^\tau\int_{\T^d}\overline{\left|\nabla_x\Phi_{\overline{\rho}} - \nabla_x\Phi_r\right|^2}\cdot\DIV_xU\diff x\diff t\\
&\frac{c}{\eps}\int_0^\tau\int_{\T^d}(\overline{\nabla_x\rho - \nabla_x r) \otimes(\nabla_x\rho - \nabla_x r)} : \nabla_x U\diff x\diff t, \quad \frac{c}{2\eps}\int_0^\tau\int_{\T^d}\overline{|\nabla_x\rho - \nabla_x r|^2}\cdot \DIV_x U\diff x\diff t\\
&\frac{b}{\eps}\int_0^\tau\int_{\T^d}\overline{\left(\nabla_x\Phi_{\overline{\rho}} - \nabla_x\Phi_r\right) \otimes \left(\nabla_x\Phi_{\overline{\rho}} - \nabla_x\Phi_r\right)} : \nabla_x U\diff x\diff t
\end{align*}
by $\int_0^\tau \mathcal{E}_{\text{rel}}(t)\diff t$ are obvious. Here, we want to emphasize again, that if $c = 0$, then for the following argument to hold it is enough for $r$, to belong to $C^2((0, T)\times\T^d)$, which is guaranteed from Appendix \ref{appendixA} and \cite{cieslak2008finite}. To handle the next term we see that (look at \eqref{lem:definitionofh} and \eqref{relativepressuredef})
$$
\overline{p(\rho|r)} = \frac{1}{\gamma - 1}\overline{h(\rho|r)} \lesssim \overline{h(\rho|r)}.
$$
Hence, also
$$
\left|\frac{1}{\eps}\int_0^\tau\int_{\T^d}\DIV_x U \,\overline{p(\rho|r)}\diff x\diff t\right| \lesssim \int_0^\tau \mathcal{E}_{\text{rel}}(t)\diff t.
$$
By Young's inequality
$$
\left|\frac{c}{\eps}\int_0^\tau\int_{\T^d}\nabla_x(\DIV_xU)\overline{(\rho - r)\nabla_x(\rho - r)}\diff x\diff t\right| \lesssim \int_0^\tau \int_{\T^d}\overline{|\rho - r|^2} + \overline{|\nabla_x\rho - \nabla_x r|^2}\diff x\diff t,
$$
and with the use of Lemma \ref{ineq:inequalityforhrhor} (remember that we assume $\gamma \geq 2$)
$$
\int_0^\tau \int_{\T^d}\overline{|\rho - r|^2} + \overline{|\nabla_x\rho - \nabla_x r|^2}\diff x\diff t \lesssim \int_0^\tau\int_{\T^d}\overline{h(\rho|r)} + \overline{|\nabla_x\rho - \nabla_x r|^2}\diff x\diff t \lesssim \int_0^\tau \mathcal{E}_{\text{rel}}(t)\diff t.
$$
Similarly 
\begin{align*}
\left|\frac{b}{\eps}\int_0^\tau\int_{\T^d}(M_{\overline{\rho}} - M_r)(\overline{\nabla_x\Phi_{\overline{\rho}}} - \nabla_x\Phi_r)\, U\diff x\diff t\right|&\lesssim \int_0^\tau \int_{\T^d}\overline{|\rho - r|^2} + \overline{|\nabla_x\Phi_{\overline{\rho}} - \nabla_x\Phi_r|^2}\diff x\diff t\\
&\lesssim \int_0^\tau \mathcal{E}_{\text{rel}}(t)\diff t.
\end{align*}
The last term left to bound in the right-hand side of \eqref{ineq:relative4} is the trickiest, therefore let us argue more carefully. Note that by Young's inequality and Proposition \ref{app:concetrationmeasuresbounding}
\begin{align*}
&\int_0^\tau\int_{\T^d}\frac{e}{r}\left(\overline{\rho u} - \overline{\rho U}\right)\diff x\diff t = \int_0^\tau\int_{\T^d}\frac{e}{r}\langle\nu_{t,x}, s(v - U)\rangle + m^{\rho u}\diff x\diff t\\
&\leq \int_0^\tau\int_{\T^d}\frac{\eps^2}{2}\left(\left|\frac{e}{r}\right|^2\langle\nu_{t,x}, s\rangle\right) + \frac{1}{2\eps^2}\left(\langle\nu_{t,x}, s|v - U|^2\rangle + m^{\rho |u|^2}\right)\diff x\diff t\\
&=  \frac{\eps^2}{2}\int_0^\tau\int_{\T^d}\left|\frac{e}{r}\right|^2\overline{\rho} + \frac{1}{2\eps^2}\int_0^\tau\int_{\T^d}\overline{\rho|u - U|^2} \diff x\diff t.
\end{align*}
As $e = O(\eps)$ we obtain
\begin{align*}
    \frac{\eps^2}{2}\int_0^\tau\int_{\T^d}\left|\frac{e}{r}\right|^2\overline{\rho} + \frac{1}{2\eps^2}\int_0^\tau\int_{\T^d}\overline{\rho|u - U|^2} \diff x\diff t \leq O(\eps^4) + \frac{1}{2\eps^2}\int_0^\tau\int_{\T^d}\overline{\rho|u - U|^2} \diff x\diff t.
\end{align*}
It is easy to check that the proven bounds, together with \eqref{ineq:relative4} end the proof of the proposition.
\end{proof}
\noindent Having Theorem \ref{prop:relentropyinequality} the proof of the main one is a formality.
\begin{proof}\textup{\textbf{(Of Theorem \ref{theo:poissonlimit})}}\\
By the Gr\"{o}nwall's inequality and Theorem \ref{prop:relentropyinequality} we obtain
$$
\mathcal{E}^\varepsilon_{\text{rel}}(\tau) \leq (O(\eps^4) + \mathcal{E}^\varepsilon_{\text{rel}}(0))e^T .
$$
Since we already assume 
$$
\mathcal{E}^\varepsilon_{\text{rel}}(0) \longrightarrow 0\quad \text{ as }\varepsilon\to 0^+,
$$
therefore
$$
\mathcal{E}^\varepsilon_{\text{rel}}(\tau) \longrightarrow 0.
$$
Going back to the inequality in Theorem \ref{prop:relentropyinequality} we obtain
$$
\frac{1}{2\eps^2}\int_0^\tau\int_{\T^d}\overline{\rho|u - U|^2} \diff x\diff t \longrightarrow 0
$$
and the conclusion follows.
\end{proof} 

\begin{rem}
In case of $c = 0$ one could show versions of both Theorem \ref{prop:relentropyinequality} and Theorem \ref{theo:poissonlimit} with the assumption that either
\begin{itemize}
    \item $\frac{2}{K} > b > 0$, $\gamma > 2 - \frac{2}{d}$ (where $K$ is taken from Lemma \ref{lem:boundfornegativebees}) or
    \item $b\leq 0$, $\gamma > 1$,
\end{itemize}
but it would result in the need of a strong solution in $C^{3, 1}(\T^d\times [0, T]))$, whose existence is unknown. See the proof of Theorem 3.9 in \cite{lattanzio2017fromgas} for the details.
\end{rem}

\appendix

\section{Strong solutions to chemo-repulsive Keller-Segel equations.}\label{appendixA}

\begin{thm}
Assume that $\rho_0\in C^{2 + \beta}(\T^d)$ for some $\beta\in (0, 1)$ and
$$
0 < \rho_* \leq \rho\leq \rho^* < +\infty.
$$
Then, there exists a constant $\alpha(\rho_*, \rho^*, \beta)\in (0, 1)$ and a function $\rho \in C^{2 + \alpha, 1+ \frac{\alpha}{2}}(\T^d\times [0, T])$ such that
\begin{equation}\label{strongsolutions:theequation}
    \begin{split}
        \partial_t\rho - \operatorname{div}_x(\nabla_x(\rho^\gamma) + \rho\nabla_x \Phi_{\rho})&=0,\\
		-\Delta_x\Phi_{\rho}&=\rho-M_{\rho}.
    \end{split}
\end{equation}
\end{thm}
\begin{proof}
We will proceed by using a Schauder's fixed point argument. For now let us fix time $\tau = \frac{1}{2\rho^*e}$, which without loss of generality is less than $T$. Let
\begin{align}
    X := \{w\in C^{\alpha, \frac{\alpha}{2}}(\T^d\times [0, \tau])\quad :\quad \eta\leq w\leq M,\quad |w|_{\alpha, \frac{\alpha}{2}}\leq N\},
\end{align}
where $\alpha\in (0, 1)$, $\eta >0$, $M >0$ and $N > 0$ are constants to be specified later, $X$ is considered with a standard norm of $C^{\alpha, \frac{\alpha}{2}}(\T^d\times [0, \tau])$ space and $|\cdot|_{\alpha, \frac{\alpha}{2}}$ is a standard H\"{o}lder seminorm. Consider a function $f$ defined by
\begin{align*}
    f(\rho)=\left\{\begin{array}{ll}
         \gamma \rho^{\gamma - 1},\quad&\text{ when } \rho \geq \eta,\\
         \text{convex and smooth},\quad&\text{ when } \frac{\eta}{2} \leq \rho\leq \eta,\\
         \frac{\eta}{2},\quad&\text{ when } 0 < \rho \leq \frac{\eta}{2},
    \end{array}
    \right.
\end{align*}
and an equation
\begin{equation}\label{strongsolutions:schauderequation}
    \begin{split}
        \partial_t\rho - \operatorname{div}_x(f(\rho)\nabla_x\rho + \rho\nabla_x \Phi_{w})&=0,\\
		-\Delta_x\Phi_{w}&=w-M_{w},\\
		\int_{\T^d}\Phi_{w}\diff x &= 0.
    \end{split}
\end{equation}
We wish to define an operator $F: X\ni w\longmapsto \rho\in X$, where $\rho$ is a solution to \eqref{strongsolutions:schauderequation}, and show that it satisfies Schauder's fixed point theorem. We need to check that $F$ is well-defined, i.e. $F(w) = \rho$ is a solution to \eqref{strongsolutions:schauderequation}, that $F$ is continuous, maps $X$ to $X$ and $F(X)$ is relatively compact in $X$.\\
\\
\underline{Step 1: F is well-defined, \`{a} priori bounds of the solution.} We can rewrite the first equation of \eqref{strongsolutions:schauderequation} as
\begin{align}\label{strongsolutions:firstequation}
    \p_t\rho - \DIV_x(A(x, t, \rho, \nabla_x\rho)) = 0,
\end{align}
where $A_i(x, t, \rho, \nabla_x\rho) = f(\rho)\frac{\p\rho}{\p x_i} + \rho\frac{\p\Phi_{w}}{\p x_i}$. We wish to use \cite[Chapter XII, Theorem 12.14]{lieberman1996second}. But first, let us notice that by \cite[Chapter 3, Theorem 18]{friedman1964partial} $\Phi_w$ has $C^{2+\alpha}$ regularity with respect to spatial variable. Moreover, for example by \cite[Chapter 4, Lemma 4.2]{gilbarg2001elliptic}, its gradient is given by the convolution of the gradient of a Newtonian potential together with the function $w$, hence it has $C^{\frac{\alpha}{2}}$ regularity in time variable. Now, we shall check assumptions of the mentioned Theorem $12.14$.
\begin{itemize}
    \item
    $$
    \frac{\eta}{2}|\xi|^2\leq f(\rho)|\xi^2| = \sum_{i, j}\frac{\p A_i(x, t, \rho, p))}{\p p_j}\xi_i\xi_j,
    $$
    \item
    $$
    p\,A(x, t, \rho, p) = f(\rho)|p|^2 + \rho\,p\,\nabla_x\Phi_w \geq \frac{\eta}{2}|p|^2 - C\left(\varepsilon|p|^2 + \frac{1}{\varepsilon}|\rho|^2\right)
    $$
    taking $\varepsilon$ small enough so that $\frac{\eta}{2} - C\varepsilon >0$ one gets what is needed,
    \item $|p|^2\,|A_p| + |p||A_\rho| + |A_x| = |f(\rho)|\,|p|^2 + |f'(\rho)|\,|p|^2 + |\nabla_x\Phi_w||p| + |\rho\,\nabla_x^2\Phi_w| = O(|p|^2)$,
    \item $A_x$, $A_{\rho}$, $A_p$ are H\"{o}lder continuous with respect to the $x$, $t$, $\rho$, $p$ variables with exponents $\alpha$, $\frac{\alpha}{2}$, $\alpha$, $\alpha$ respectively.
\end{itemize}
Hence, there exists a unique function $\rho\in C^{2 + \alpha, 1 + \frac{\alpha}{2}}(\T^d\times [0, \tau])$ satisfying \eqref{strongsolutions:schauderequation}. Moreover by \cite[Chapter XII, Theorem 12.10]{lieberman1996second} and \cite[Chapter V, Theorem 1.1]{ladyzhenskaya1968linear} there exists a $\alpha(\beta, \eta, M, \rho^*)$ such that
\begin{align}\label{strongsolutions:holderboundsfirstderivative}
    \|D_{x, t}\rho\|_{C^{\alpha, \frac{\alpha}{2}}}\leq C(\beta, \eta, M, \rho^*, |w|_{\alpha, \frac{\alpha}{2}})
\end{align}
and
\begin{align}\label{strongsolutions:holderbounds}
    |\rho|_{\alpha, \frac{\alpha}{2}}\leq C(\beta, \eta, M, \rho^*).
\end{align}
Combining \eqref{strongsolutions:holderbounds} and \eqref{strongsolutions:holderboundsfirstderivative} one may treat \eqref{strongsolutions:firstequation} as a linear equation and obtain from linear theory (see comments below \cite[Chapter XII, Theorem 12.14]{lieberman1996second} and \cite[Chapter V, Theorem 5.14]{lieberman1996second})
\begin{align}\label{strongsolutions:boundsforarzela}
    \|\rho\|_{C^{2+\alpha, 1+\frac{\alpha}{2}}}\leq C(\beta, \eta, M, \rho^*, |w|_{\alpha, \frac{\alpha}{2}}).
\end{align}
\\
\underline{Step 2: F maps $X$ into $X$.} We already know from the previous step that $\rho\in C^{\alpha, \frac{\alpha}{2}}(\T^d\times [0, \tau])$ and if we fix $N: = C(\beta, \eta, M, \rho^*)$ from \eqref{strongsolutions:holderbounds}, then all that is left to prove is that 
\begin{align}\label{strongsolutions:boundsforrhoschauder}
\eta\leq \rho \leq M.
\end{align}
To this end, let $\max_{x}\rho(x, t) = \rho(x_0, t)$. Then, putting it in \eqref{strongsolutions:firstequation} one obtains
$$
\frac{d}{dt}\max_{x}\rho(x, t) = \max_x\rho(x, t)\,\Delta_x\Phi_w + f(\rho)\Delta_x\rho(x_0, t) \leq  \max_x\rho(x, t)\,\Delta_x\Phi_w \leq 2M\max_x\rho(x, t).
$$
Hence
$$
\max_{x}\rho(x, t) \leq \rho^* \exp(2M\tau).
$$
Proceeding similarly for the minimum we get
$$
\min_x\rho(x, t) \geq \rho_*\exp(-2M\tau).
$$
Setting $M = \rho^*\,e$ and $\eta = \rho_*\,e^{-1}$ and remembering that $\tau = \frac{1}{2\,\rho^*\,e}$ one gets \eqref{strongsolutions:boundsforrhoschauder}.\\
\\
\underline{Step 3: F(X) is relatively compact.} This fact follows from \eqref{strongsolutions:boundsforarzela} and Arzel\'{a}--Ascoli theorem.\\
\\
\underline{Step 4: F is continuous.} Take a sequence $w_n$ and its limit $w$ in $X$. Then, from \eqref{strongsolutions:boundsforarzela} and Arzel\'{a}--Ascoli theorem we know that $\rho_n:=F(w_n)$ is relatively compact in $C^{2, 1}(\T^d\times[0, \tau])$ and $\rho_{n_k} \rightarrow \rho$ for some $\rho\in C^{2, 1}(\T^d\times[0,\tau])$. Hence, one can converge with the subsequence in \eqref{strongsolutions:firstequation} to obtain
$$
\p_t\rho - \DIV_x A(x, t, \rho, \nabla_x\rho) = 0.
$$
By the uniqueness of the solution, we get $\rho = F(w)$. A standard subsequence argument ensures the continuity of $F$.\\
\\
With this, we have completed the proof of the existence of the local solution to \eqref{strongsolutions:theequation}. Now assume that there exists a maximal interval of existence $[0, t_0)$ with $t_0 < T$. Similarly as before, let $\max_x\rho(x, t) = \rho(x_0, t)$. Then, putting it into the first equation of \eqref{strongsolutions:theequation}
\begin{multline*}
\frac{d}{dt}\max_x\rho(x, t) = \max_{x}\rho(x, t)\Delta_x\Phi_\rho + \gamma\rho^{\gamma - 1}\Delta_x\rho\\
\leq \max_x\rho(x,t)\Delta_x\Phi_\rho = -\rho(x_0, t)^2 + M_{\rho}\,\rho(x_0, t)\leq M_{\rho}\max_x\rho(x,t).
\end{multline*}
Thus
\begin{align}\label{strongsolutions:globalboundsup}
    \max_x{\rho(x, t)}\leq \rho^*\exp(\rho^*\,T).
\end{align}
Moving forward, fix $\min_x\rho(x, t) = \rho(x_0, t)$ and put it into the first equation of \eqref{strongsolutions:theequation}
\begin{multline*}
    \frac{d}{dt}\min_{x}\rho(x, t) = \min_{x}\rho(x, t)\Delta_x\Phi_\rho + \gamma\rho^{\gamma - 1}\Delta_x\rho \geq -\rho(x_0, t)^2 + M_{\rho}\,\rho(x_0, t) \geq -(M + M_{\rho})\rho(x_0, t).
\end{multline*}
Hence
\begin{align}\label{strongsolutions:globalboundsdown}
    \min_x\rho(x, t) \geq \rho_*\exp(-(M+\rho^*)T).
\end{align}
Both of the bounds \eqref{strongsolutions:globalboundsup} and \eqref{strongsolutions:globalboundsdown} are global, meaning that $\lim_{t\to t_0}\rho(x, t)$ is an admissible initial condition and we can repeat the steps above to obtain a solution on the interval $[t_0, t_0 + \delta)$ for some $\delta > 0$, which is a contradiction with maximality of the interval $[0, t_0)$ and ends the proof.
\end{proof}

\bibliography{poisskort}
\bibliographystyle{abbrv}

\end{document}